\documentclass[a4paper]{amsart}
\usepackage{graphicx}
\usepackage{amssymb}
\usepackage{amsmath}
\usepackage{amsthm}
\usepackage{color}
\usepackage{colordvi}
\usepackage{tabularx}
\usepackage[all,2cell]{xy}
\usepackage{float}
\usepackage{hyperref}

\UseAllTwocells \SilentMatrices
\newtheorem{thm}{Theorem}[section]

\newtheorem{lem}[thm]{Lemma}

\newtheorem{prop}[thm]{Proposition}
\theoremstyle{definition}
\newtheorem{defn}[thm]{Definition}
\newtheorem{hypo}{Hypothesis}

\theoremstyle{remark}
\newtheorem{rem}[thm]{\bf Remark}
\numberwithin{equation}{section}

\newcommand{\eps}{\varepsilon}
\newcommand{\To}{\longrightarrow}

\newcommand{\KK}{\mathcal{K}}
\newcommand{\TT}{\mathcal{T}}

\newcommand{\co}{\mathrm{co}}
\newcommand{\op}{\mathrm{op}}
\newcommand{\upH}{\mathbf{H}}
\newcommand{\Hom}{\operatorname{Hom}}
\newcommand{\Ext}{\operatorname{Ext}}
\newcommand{\upHH}{\mathbf{HH}}

\newcommand{\Ima}{\operatorname{Im}}
\newcommand{\tint}{\textstyle{\int}}

\allowdisplaybreaks

\begin{document}
\title[Cup products on Hochschild cohomology of Hopf--Galois extensions]{Cup products on Hochschild cohomology of Hopf--Galois extensions}

\author[Liyu Liu]{Liyu Liu}
\address{School of Mathematical Sciences, Yangzhou University, No.\ 180 Siwangting Road, Yangzhou 225002, Jiangsu, PR China}
\email{lyliu@yzu.edu.cn}

\author[Wei Ren]{Wei Ren}
\address{School of Mathematical Sciences, Chongqing Normal University, Chongqing 401331, PR China}
\email{wren@cqnu.edu.cn}

\author[Shengqiang Wang]{Shengqiang Wang}
\address{School of Mathematics, East China University of Science and Technology, Shanghai 200237, PR China}
\email{sqwang@ecust.edu.cn}

\thanks{}


\begin{abstract}
In this paper, we give an explicit chain map, which induces the algebra isomorphism between the Hochschild cohomology  $\upHH^{\bullet}(B)$ and the $H$-invariant subalgebra $\upH^{\bullet}(A, B)^{H}$ under two mild hypotheses, where $H$ is a finite dimensional semisimple Hopf algebra and $B$ is an $H$-Galois extension of $A$. In particular, the smash product $B=A\#H$ always satisfies the mild hypotheses. The isomorphism between $\upHH^{\bullet}(A\#H)$ and  $\upH^{\bullet}(A, A\#H)^{H}$ generalizes the classical result of group actions. As an application, Hochschild cohomology and cup product of the smash product of the quantum $(-1)$-plane and Kac--Paljutkin Hopf algebra are computed.
\end{abstract}

\maketitle

\dedicatory{}
\commby{}

\section{Introduction}

Let $\Bbbk$ be a field and $A$ be a $\Bbbk$-algebra. The Hochschild cohomology $\upH^{\bullet}(A, N)$ of $A$ with coefficients in an $A$-bimodule $N$ was introduced in \cite{Hoc45} in order to classify, up to equivalence, all extensions of $A$ with kernel $N$. In particular, $\upHH^\bullet(A):=\upH^\bullet(A,A)$ admits an additional structure found by Gerstenhaber \cite{Ger63}, under which we now say that Hochschild cohomology is a Gerstenhaber algebra. Roughly speaking, a Gerstenhaber algebra is an $\mathbb{N}$-graded vector space equipped with two binary operations, cup product and Gerstenhaber bracket, which satisfy several axioms. The Gerstenhaber algebra structure on $\upHH^\bullet(A)$ was essentially based on the bar complex of $A$ in \cite{Ger63}. The bar complex, as an $A^e$-projective resolution of $A$, is very big, so in practice one seldom computes the structure via the bar complex directly. Later work invokes many other projective resolutions, depending on the setting.

Let $A$ be an algebra upon which a finite group $G$ acts by automorphisms, and $A\#G$ denote the resulting skew group algebra. In the commutative case, for polynomial algebra $A$, Farinati \cite{Far05}, Ginzburg and Kaledin \cite{GK04} proved independently that $\upHH^{\bullet}(A\#G)\cong\upH^{\bullet}(A, A\#G)^{G}$ as vector spaces, where the supscript $G$ denotes $G$-invariants. Based on this, Shepler and Witherspoon investigated the cup product structure on $\upHH^\bullet (A\#G)$ in \cite{SWi11}. Later, Negron and Witherspoon also described the Gerstenhaber brackets on $\upHH^\bullet (A\#G)$ in \cite[Theorem 5.2.3]{NW17}.

As for the noncommutative algebra $A$ case, Witherspoon and Zhou proved the aforementioned Gerstenhaber algebra isomorphism for quantum symmetric algebras with diagonal group extensions \cite{WZ16}. Burciu and Witherspoon proved an isomorphism $\upH^\bullet(A\#H, -)\cong\Ext_{\Gamma}^\bullet(A,-)$ for smash products, where $H$ is a Hopf algebra and $\Gamma$ is a subalgebra of $(A\# H)^e$ \cite{BW07}. Negron gave a multiplicative spectral sequence to compute the cup product of $\upHH^\bullet (A\#H)$ \cite{Neg15}. We remark that the projective resolution constructed by Negron is in fact a $\Gamma$-module complex (ibid.). Wang succeeded in finding $\Gamma$ for $H$-Galois extensions $B$ over $A$ \cite{Wang18} and obtained a generalized isomorphism $\upH^\bullet(B, -)\cong\Ext_{\Gamma}^\bullet(A,-)$. Witherspoon and her collaborators also presented techniques for computing Gerstenhaber brackets on Hochschild cohomology of general twisted tensor product algebras. These techniques involve twisted tensor product resolutions and are based on Gerstenhaber brackets expressed via arbitrary bimodule resolutions \cite{KMO21}. Furthermore, Briggs and Witherspoon studied the Hochschild cohomology of twisted tensor products in \cite{BW22}, and proved the Gerstenhaber algebra isomorphism for any finite dimensional algebra $A$ with an action of finite abelian group $G$. Besides, many researchers computed lots of examples on this topic.

It is worth noting that Stefan's homological and cohomological spectral sequences are powerful tools for investigating homology and cohomology of $H$-Galois extensions \cite{Ste95}. The original proof is to employ universal $\delta$-functors and Grothendieck spectral sequences. In contrast to that, when restricted to smash product situation, the cohomological spectral sequence has been proved by double complex \cite{Neg15}. It implies $\upHH^{\bullet}(A\#H)\cong\upH^{\bullet}(A, A\#H)^{H}$ for finite dimensional semisimple Hopf algebras $H$. The isomorphism can be induced by appropriate (co)chain maps in different contexts (cf. \cite{Neg15}, \cite{SW12}, etc.). In this paper, we intend to unify and generalize the various cases by considering Hopf--Galois extensions, which include skew group algebras, smash products, crossed products and so on. We will reprove Stefan's spectral sequences using double complex. Consequently, under two mild hypotheses, there still exists a chain map on the level of complexes, which induces an isomorphism on the cohomologies.

\begin{thm}[Thm.\ \ref{thm:cup-product-semisimple}]\label{thm:main-1}
	Let $B/A$ be an $H$-Galois extension where $H$ is semisimple. Suppose that Hypotheses I and II given in Sec.\ \ref{sec:cup-semisimple} hold true. Then there is an isomorphism $\upHH^{\bullet}(B)\cong\upH^{\bullet}(A, B)^{H}$ that preserves cup products.
\end{thm}

This makes computation of cup products available for concrete algebras, especially for smash products $A\# H$. Although Negron has given beautiful formulas in \cite{Neg15} to determine cup product structure on $\upHH^\bullet(A\#H)$, he also posed two additional conditions for $A^e$-projective resolutions of $A$ (different from ours). So we focus on smash products over Koszul algebras, and apply the isomorphism in Theorem \ref{thm:main-1}. In this case, we show that the Koszul complex of $A$ satisfies the two hypotheses and establish a relation between $\upHH^\bullet(A\#H)$ and the Koszul dual $A^!$.

\begin{thm}[Thm.\ \ref{thm:cup-koszul}]\label{thm:main-2}
	The complexes of right $H$-modules, $A^!\otimes (A\# H)$ and $\Hom_{A^e}(K(A), A\# H)$, are mutually isomorphic. Consequently, $\upHH^\bullet(A\# H)$ is isomorphic to the $H$-invariant subalgebra of $\upH^\bullet(A^!\otimes (A\# H))$ as a graded algebra.
\end{thm}

As a generalization of group actions on polynomial algebras, Hopf algebra actions on Artin--Schelter regular algebras are attractive in the field of noncommutative algebraic geometry. We consider an action of Kac--Paljutkin Hopf algebra on the quantum $(-1)$-plane, and compute the Hochschild cohomology of the resulting smash product, as well as the cup product, by applying Theorem \ref{thm:main-2}.

The paper is organized as follows: We briefly recall some basic materials on Hochschild (co)homology, Hopf--Galois extensions and Koszul algebras in Section \ref{sec:preliminaries}. We then reprove Stefan's spectral sequences by providing some technical lemmas in Section \ref{sec:stefan-sequence}. Our main results are in Section \ref{sec:cup-semisimple} where we prove that under two mild hypotheses, $\upHH^{\bullet}(B)$ and the $H$-invariant subalgebra $\upH^{\bullet}(A, B)^{H}$ are isomorphic as graded algebras, which is induced by an explicit chain map. In Section \ref{sec:hoch-smash-koszul}, we consider the special situation that $A$ is a Koszul algebra and $H$ is a finite dimensional semisimple Hopf algebra. In this case, $\upHH^{\bullet}(A\#H)$ is also isomorphic to the $H$-invariant subalgebra of $\upH^\bullet(A^!\otimes (A\#H))$ as graded algebras, where $\upH^\bullet(A^!\otimes (A\#H))$ is the cohomological algebra of the differential graded algebra $A^!\otimes (A\#H)$. We illustrate this result with the smash product of quantum $(-1)$-plane and Kac--Paljutkin Hopf algebra in Section \ref{sec:kac}.

\section{Preliminaries}\label{sec:preliminaries}

Throughout this paper, $\Bbbk$ is a field and all algebras are over $\Bbbk$. Unadorned $\otimes$ and $\Hom$ stand for $\otimes_\Bbbk$ and  $\Hom_\Bbbk$ respectively. For an algebra $A$, $A^{\op}$ is the opposite algebra of $A$, and $A^e:=A\otimes A^{\op}$ is the enveloping algebra of $A$. Thus all $A$-bimodules are identified with left or right $A^e$-modules naturally. For any algebraic object (vector space, algebra, complex, etc.) $U$, denote by $I_U$ the identity map of $U$; we often drop the subscript $U$ if $U$ is definite from the context.

\subsection{Hochschild cohomology}
Let $N$ be an $A$-bimodule. The $m$-th Hochschild cohomological group of $A$ with coefficients in $N$ is $\upH^{m}(A, N):= \Ext^{m}_{A^{e}}(A, N)$. One useful $A^{e}$-projective resolution of $A$ is the bar resolution $B(A)$:
\[
\cdots \stackrel{\delta_{3}}\longrightarrow A^{\otimes 4}\stackrel{\delta_{2}}\longrightarrow A^{\otimes 3}\stackrel{\delta_{1}}\longrightarrow A^{\otimes 2},
\]
where $\delta_{m}(a_{0}\otimes \cdots \otimes a_{m+1})= \sum_{i=0}^{m}(-1)^{i}a_{0}\otimes \cdots \otimes a_{i}a_{i+1}\otimes \cdots \otimes a_{m+1}$ and the augmentation map $\delta_0\colon A^{\otimes 2} \to A$ is the multiplication. Using the bar resolution, $\upH^{\bullet}(A, N):=\bigoplus_{m\in\mathbb{N}}\upH^{m}(A, N)$ is the cohomology of the complex $C^\bullet(A, N)$:
\[
0\To \Hom(\Bbbk, N) \overset{\delta^1}{\To} \Hom(A, N) \overset{\delta^2}{\To} \Hom(A^{\otimes 2}, N) \overset{\delta^3}{\To} \cdots,
\]
whose differentials are given by
\begin{align*}
\delta^m(f)(a_{1},\ldots, a_{m+1})&=a_1f(a_{2}, \ldots, a_{m+1})+\sum_{i=1}^m(-1)^{i} f(a_{1}, \ldots, a_{i}a_{i+1}, \ldots, a_{m+1}) \\
&\mathrel{\phantom{=}}{} +(-1)^{m+1}f(a_{1}, \ldots, a_{m})a_{m+1},
\end{align*}
where for any $m$-cochain $f$, we write $f(a_1,\ldots, a_m)$ instead of $f(a_1\otimes\cdots\otimes a_m)$. Denote $\upHH^{\bullet}(A):=\upH^{\bullet}(A, A) $ if $N=A$.

Recall that the tensor coalgebra $\mathsf{T}^c(A)$ of $A$ is defined to be $\bigoplus_{m\in\mathbb{N}} A^{\otimes m}$ whose comultiplication is give by
\begin{align*}
1&\mapsto 1\otimes 1, \\
a_1\otimes\cdots\otimes a_m &\mapsto 1\otimes(a_1\otimes\cdots\otimes a_m)+(a_1\otimes\cdots\otimes a_m)\otimes 1 \\
& \mathrel{\phantom{\mapsto}} {}+\sum_{i=1}^{m-1} (a_1\otimes\cdots\otimes a_i)\otimes (a_{i+1}\otimes\cdots\otimes a_m).
\end{align*}
Hence $\mathsf{T}^c(A)$ is a graded coalgebra with the degree $m$ component $\mathsf{T}^c(A)_m=A^{\otimes m}$. Notice that $C^{\bullet}(A, N)=\underline{\Hom}(\mathsf{T}^c(A), N):=\bigoplus_{m\in\mathbb{N}}\Hom(\mathsf{T}^c(A)_m, N)$, and thus $C^{\bullet}(A, N)$ admits an associative convolution that preserves grading if $N$ has an algebra structure. In particular, let $A\to B$ be an algebra map. Then $B$ is an $A$-bimodule and the convolution on $C^\bullet(A, B)$ is called the cup product in the literature, written as $\smallsmile$, which is defined by
\begin{equation}\label{eq:cup-original}
(f\smallsmile g)(a_1,\ldots, a_{m+n})=f(a_1,\ldots, a_m)g(a_{m+1},\ldots, a_{m+n})
\end{equation}
for all $m$-cochain $f$ and $n$-cochain $g$. The cup product makes $C^\bullet(A, B)$ into a differential graded algebra, and hence $\upH^\bullet(A, B)$ is a graded algebra.

The cup product can be defined via an arbitrary projective $A$-bimodule resolution $P$ of $A$ (cf.\ \cite{Vol}). Recall that the tensor product complex $P\otimes_A P$ is defined as $(P\otimes_A P)_m=\bigoplus_{i+j=m}P_i\otimes_A P_j$ with differential
\[
d_{P\otimes_A P}(x\otimes_A y)=d_P(x)\otimes_A y+(-1)^i x\otimes_A d_P(y)
\]
for all $x\in P_i$ and $y\in P_j$. There exists a comparison $T\colon P\to P\otimes_A P$ of $A$-bimodule complexes which lifts the canonical isomorphism $A\cong A\otimes_A A$. For any $f\in\Hom_{A^e}(P_m, B)$ and $g\in\Hom_{A^e}(P_n, B)$, we regard $f$, $g$ as the $A$-bimodule homomorphisms $P\to B$ by setting $f|_{P_{m'}}=0=g|_{P_{n'}}$ for all $m'\neq m$ and $n'\neq n$, and we define $f\smallsmile g\in \Hom_{A^e}(P_{m+n}, B)$ by
\[
f\smallsmile g=\mu\circ(f\otimes_A g)\circ T
\]
where $\mu$ stands for the multiplication of $B$. Notice that $(T\otimes_A I)\circ T$ and $(I\otimes_A T)\circ T$ are homotopic (but not equal in general), so $\smallsmile$ is an associative multiplication at the level of cohomology, namely, each pair $(P, T)$ gives rise to a cup product $\smallsmile$ on $\upH^\bullet(A, B)$. Furthermore, $\smallsmile$ turns out to be independent of the choice of $(P, T)$, which coincides with \eqref{eq:cup-original}.

\subsection{Hopf--Galois extensions}

Let $H$ be a Hopf algebra. We use standard notations borrowed from  \cite{Mon93}. As usual, the comultiplication, the counit, and the antipode of $H$ are denoted by $\Delta$, $\eps$, and $S$ respectively. In this paper, $S$ is always assumed to be bijective. We call $B/A$ to be an $H$-extension if $B$ is a right $H$-comodule algebra with the structure map $\rho\colon B\to B\otimes H$, $b\mapsto \sum b_{[0]}\otimes b_{[1]}$, and $A$ is the subalgebra of $H$-coinvariants of $B$, namely, $A=B^{\co H}=\{b\in B\,|\, \rho(b)=b\otimes 1\}$.

For an $H$-extension $B/A$,  $B^{\op}$ is a left $H$-comodule algebra with the structure map  $\rho'\colon B^{\op}\to H\otimes B^{\op}$, $b\mapsto \sum S^{-1}b_{[1]}\otimes b_{[0]}$,  and we have $A^{\op}={}^{\co H}(B^{\op})$ furthermore. It is direct to check that
\begin{gather*}
\gamma:=I\otimes_A \rho\colon B\otimes_A B\to B\otimes_A B \otimes H, \\
\gamma':=\rho'\otimes_A I\colon  B^{\op}\otimes_A B\to H\otimes B^{\op}\otimes_A B
\end{gather*}
are well-defined maps, making $B\otimes_A B$ into a right $H$-comodule and $B^{\op}\otimes_A B$ into a left $H$-comodule respectively.

Associated to an $H$-extension $B/A$  there is a canonical map $\beta\colon B\otimes_A B\to B\otimes H$, $b\otimes_A b'\mapsto\sum bb'_{[0]}\otimes b'_{[1]}$. The $H$-extension $B/A$ is said to be
\begin{enumerate}
	\item flat if $B$ is flat over $A$, i.e., $B$ is a flat left and a flat right $A$-module;
	\item Galois if the canonical map $\beta$ is bijective.
\end{enumerate}
Define $\tau\colon B\otimes H\to H\otimes B$, $b\otimes h\mapsto \sum hS^{-1}b_{[1]}\otimes b_{[0]}$. This is bijective with $\tau^{-1}(h\otimes b)= \sum b_{[0]}\otimes hb_{[1]}$. The composition $\beta'=\tau\circ\beta\colon B\otimes_A B\to H\otimes B$ is given by $\beta'(b\otimes_A b')= \sum S^{-1}b_{[1]}\otimes b_{[0]}b'$. So $B/A$ is Galois if and only if $\beta'$ is bijective.

Let $B/A$ be an $H$-Galois extension. Recall that for any $h\in H$, $\beta^{-1}(1\otimes h)$ is denoted by $\sum_i l_i(h)\otimes_A r_i(h)$ by convention. It follows that $(\beta')^{-1}(h\otimes 1)=\beta^{-1}(1\otimes h)=\sum_i l_i(h)\otimes_A r_i(h)$.  Since $\beta$ and $\beta'$ are right and left $H$-comodule homomorphisms respectively, one has the following equations:
\begin{align}
& \sum_il_i(h)r_i(h)=\eps(h), \label{eq:galois-1}\\
& \sum_i \gamma(l_i(h)\otimes_A r_i(h))=\sum_i l_i(h_{(1)})\otimes_A r_i(h_{(1)})\otimes h_{(2)}, \label{eq:galois-2}\\
& \sum_i \gamma'(l_i(h)\otimes_A r_i(h))=\sum_i h_{(1)}\otimes l_i(h_{(2)})\otimes_A r_i(h_{(2)}), \label{eq:galois-3}\\
& \sum_i l_i(hh')\otimes_A r_i(hh')=\sum_{i,j} l_i(h')l_j(h)\otimes_A r_j(h)r_i(h'), \label{eq:galois-4}\\
& \sum_i al_i(h)\otimes_A r_i(h)=\sum_i l_i(h)\otimes_A r_i(h)a, \label{eq:galois-5}
\end{align}
for all $h$, $h'\in H$ and $a\in A$.

\subsection{Koszul algebras}

Let $V$ be a finite dimensional vector space, and $\mathsf{T}(V)$ be the tensor algebra with the usual grading. A graded algebra $A=\mathsf{T}(V)/(R)$ is called a quadratic algebra if $R$ is subspace of $V^{\otimes 2}$. The homogeneous dual $A^!$ of $A$ is defined as $\mathsf{T}(V^*)/(R^\perp)$, where $V^*$ is the dual space of $V$ and $R^\perp$ is the orthogonal space of $R$ in $(V^*)^{\otimes 2}$.

\begin{rem}
	In the definition of $A^!$, one should identify $(V^*)^{\otimes 2}$ with $(V^{\otimes 2})^*$. There are two manners in the literature: $(\xi_1\otimes \xi_2)(v_1\otimes v_2)=\xi_1(v_1)\xi_2(v_2)$, or $(\xi_1\otimes \xi_2)(v_1\otimes v_2)=\xi_2(v_1)\xi_1(v_2)$, where $\xi_1$, $\xi_2\in V^*$ and $v_1$, $v_2\in V$. In this paper, we adopt the former.
\end{rem}

Let $\{e_i\}_{1\leq i\leq n}$ be a basis of $V$ and $\{e^i\}_{1\leq i\leq n}$ be the dual basis of $V^*$. For any $m\geq 0$, the $m$th homogeneous component of $A^!$ is denoted by $A^!_m$, and let $(A^!)^*=\bigoplus_{m\in\mathbb{N}} (A^!_m)^*$ be the graded dual space of $A^!$. Then $(A^!)^*$ is an $A^!$-bimodule in a natural way. In particular, for each $\alpha\in (A^!_m)^*$, $e^i\alpha$ and $\alpha e^i$ as two linear maps $(A^!_{m+1})^*\to (A^!_m)^*$ are given by
\[
e^i\alpha(\xi)=\alpha(\xi e^i)\quad\text{and}\quad \alpha e^i(\xi)=\alpha(e^i\xi)
\]
respectively, for all $\xi\in A^!_m$.  There are two $A$-bimodule homomorphisms
\[
d_l, d_r\colon A\otimes(A^!_{m+1})^*\otimes A\To A\otimes(A^!_{m})^*\otimes A
\]
given by $d_l(a\otimes\alpha\otimes a')=\sum_iae_i\otimes\alpha e^i\otimes a'$ and $d_r(a\otimes\alpha\otimes a')=\sum_ia\otimes e^i \alpha \otimes e_ia'$. Let $d_K^{m}=d_l-(-1)^{m-1}d_r$. Then
\[
\cdots\To A\otimes(A^!_{m+1})^*\otimes A\overset{d_K^{m+1}}{\To} A\otimes(A^!_{m})^*\otimes A \overset{d_K^m}{\To} \cdots\overset{d_K^1}{\To} A\otimes(A^!_{0})^*\otimes A
\]
is a complex, called the Koszul bimodule complex of $A$. We denoted this complex by $K(A)$.

\begin{defn}
	A quadratic algebra $A$ is called Koszul if the trivial left $A$-module $\Bbbk$ admits a projective resolution $P$ such that $P_m$ is generated in degree $m$ for all $m\geq 0$.
\end{defn}

A quadratic algebra $A$ is Koszul if and only if $K(A)$ is a resolution of $A$ via the multiplication $A\otimes (A^!_{0})^*\otimes A\to A$. In this case, $(A^!)^*$ is finite dimensional. By identifying $W^{**}$ with $W$ for any finite dimensional vector space $W$, we have
\begin{equation}\label{eq:koszul-embedding}
(A^!_m)^*=\bigcap_{u+v=m-2} V^{\otimes u}\otimes R \otimes V^{\otimes v}  \subseteq A^{\otimes m}
\end{equation}
for all $m\geq 2$. Additionally, $(A^!_0)^*=\Bbbk$ and $(A^!_1)^*=V$. Hence $K(A)_m\subseteq A^{\otimes m+2}$. Notice that $A^{\otimes m+2}$ is the $m$-th component of the bar complex $B(A)$ of $A$, and $K(A)$ is actually a subcomplex of $B(A)$ via the above inclusion.

\begin{rem}
	For any Koszul algebra $A$, $K(A)$ is the minimal free resolution of $A$ over $A^e$.
\end{rem}

\section{Stefan's spectral sequences for Hopf--Galois extensions}\label{sec:stefan-sequence}

From now on, let $B/A$ be a flat $H$-Galois extension. In this case, $B^{\op}/A^{\op}$ is a left $H$-extension. If $A^{\op}$ is just regarded as a subspace, not a subalgebra, of $B^{\op}$,  we often drop the superscript ``op'' and write $A={}^{\co H}(B^{\op})$ for simplicity.

Notice that $(-)^{\co H}=-\square_H \Bbbk$ and ${}^{\co H}(-)=\Bbbk\square_H-$, we then have
\begin{gather}
 (B\otimes_A B)^{\co H}=B\otimes_A A, \label{eq:cotensor-1}\\
 {}^{\co H}(B^{\op}\otimes_A B)=A\otimes_A B, \label{eq:cotensor-2}\\
 (B\otimes_A B)\square_H (B^{\op}\otimes_A B)=B\otimes_A \Gamma\otimes_A B, \label{eq:cotensor-3}
\end{gather}
since $B$ is flat over $A$ and thus $B\otimes_A-$, $-\otimes_A B$ commute with taking cotensor products. Denote by $\Gamma$ the cotensor product $B\square_H B^{\op}$ (see \cite{WZ16}), namely,
\[
\Gamma=\biggl\{\sum_p x_p\otimes y_p\in B^e\,\biggm|\,\sum_p x_{p[0]}\otimes x_{p[1]}\otimes y_p=\sum_p x_p\otimes S^{-1}y_{p[1]}\otimes y_{p[0]} \biggr\}.
\]
Obviously, $\Gamma$ is a subalgebra of $B^e$ which contains $A^e$.

\begin{rem}
	$B^e$ is a right $H$-comodule via $b\otimes b'\mapsto \sum b_{[0]}\otimes b'_{[0]}\otimes b_{[1]}b'_{[1]}$. It turns out $(B^e)^{\co H}=\Gamma$ (see \cite[Lemma 2.2]{CCMT07}).
\end{rem}

\begin{rem}
	When $A$ is an $H$-module algebra and $B=A\# H$, one has $\Gamma\cong A^e\rtimes H$ (see \cite{BW07}, \cite{Kay07}).
\end{rem}

Stefan proved that there exists spectral sequences
\begin{gather*}
E^2_{pq}=\upH_p(H, \upH_q(A, N))\Longrightarrow \upH_{p+q}(B, N), \\
E_2^{pq}=\upH^p(H, \upH^q(A, N))\Longrightarrow \upH^{p+q}(B, N)
\end{gather*}
for any $B$-bimodule $N$. In his poof, the $H$-bimodule structures on $\upH_{\bullet}(A, N)$ and $\upH^{\bullet}(A, N)$  are defined by virtue of universal $\delta$-functor. We will reprove Stefan's spectral sequences by introducing $H$-bimodule structures at the level of complex, which enable us to compute homology and cohomology expediently for semisimple $H$ (see Sec.~\ref{sec:cup-semisimple}).

Let $M$ be a left $\Gamma$-module. For any $\sum_px_p\otimes y_p\in \Gamma$ and $m\in M$, we write $\sum_p x_pmy_p$ instead of $(\sum_p x_p\otimes y_p)\cdot m$ for our purposes, although the $\Gamma$-module structure on $M$ is not necessarily obtained by a $B$-bimodule structure.

\begin{lem}\label{lem:h-gamma}
	For all $h\in H$, $\sum_{i,j}l_i(h_{(1)})\otimes_A r_i(h_{(1)})\otimes l_j(h_{(2)})\otimes_A r_j(h_{(2)}) \in B\otimes_A \Gamma\otimes_A B$.
\end{lem}

\begin{proof}
	Denote $\widetilde{h}=\sum_{i,j}l_i(h_{(1)})\otimes_A r_i(h_{(1)})\otimes l_j(h_{(2)})\otimes_A r_j(h_{(2)})$. By \eqref{eq:galois-2} and \eqref{eq:galois-3} we have
	\begin{align*}
	(\gamma\otimes I\otimes_A I)(\widetilde{h})&=\sum_{i,j} \gamma(l_i(h_{(1)})\otimes_A r_i(h_{(1)}))\otimes l_j(h_{(2)})\otimes_A r_j(h_{(2)}) \\
	&=\sum_{i, j} l_i(h_{(1)})\otimes_A r_i(h_{(1)})\otimes h_{(2)} \otimes l_j(h_{(3)})\otimes_A r_j(h_{(3)}) \\
	&=\sum_{i, j} l_i(h_{(1)})\otimes_A r_i(h_{(1)})\otimes \gamma'(l_i(h_{(2)})\otimes_A r_i(h_{(2)})) \\
	&=(I\otimes_A I\otimes \gamma')(\widetilde{h}).
	\end{align*}
	Hence $\widetilde{h}\in (B\otimes_A B)\square_H (B^{\op}\otimes_A B)=B\otimes_A \Gamma\otimes_A B$ by \eqref{eq:cotensor-3}.
\end{proof}

\begin{prop}\label{prop:module-structure}
	Let $M$ be a left $\Gamma$-module and $N$ be a $B$-bimodule. Then there is a left $H$-module structure on $N\otimes_{A^e}M$ defined by
	\[
	h\rightharpoonup (n\otimes_{A^e}m)=\sum_{i,j} r_j(h_{(2)})n l_i(h_{(1)})\otimes_{A^e} r_i(h_{(1)})ml_j(h_{(2)}),
	\]
	as well as a right $H$-module structure on $\Hom_{A^e}(M, N)$ defined by
	\[
	(f\leftharpoonup h)(m)=\sum_{i,j}l_i(h_{(1)})f(r_i(h_{(1)})ml_j(h_{(2)}))r_j(h_{(2)}),
	\]
	for any $h\in H$, $m\in M$, $n\in N$, and $f\in \Hom_{A^e}(M, N)$.
\end{prop}

\begin{proof}
	It follows from Lemma \ref{lem:h-gamma} that the definitions for  $h\rightharpoonup (n\otimes_{A^e}m)$ and $f\leftharpoonup h$ are well-defined, since $\otimes$ and $\Hom$ are both over $A^e$. The fact that $f\leftharpoonup h$ is also $A$-bilinear is an immediate conclusion of \eqref{eq:galois-5}. To complete the proof, we have to show $h\rightharpoonup(h'\rightharpoonup(n\otimes_{A^e} m))=hh'\rightharpoonup(n\otimes_{A^e}m)$ and $(f\leftharpoonup h)\leftharpoonup h'=f\leftharpoonup hh'$, which follow from \eqref{eq:galois-4}.
\end{proof}

\begin{prop}\label{prop:invariant}
	Let $M$, $N$ be as above. Denote by $[H, N\otimes_{A^e}M]$ the space spanned by $h\rightharpoonup (n\otimes_{A^e}m)-\eps(h)n\otimes_{A^e}m$ for all $h$, $m$ and $n$. Then
	\begin{enumerate}
		\item $(N\otimes_{A^e} M)_H:=N\otimes_{A^e} M/[H, N\otimes_{A^e}M]$  is equal to $N\otimes_\Gamma M$, and
		\item $\Hom_{A^e}(M, N)^H:=\{f\in \Hom_{A^e}(M, N) \,|\, f\leftharpoonup h=\eps(h)f, \forall\, h\in H\}$, the space of $H$-invariants, is equal to $\Hom_\Gamma(M, N)$.
	\end{enumerate}
\end{prop}

\begin{proof}
	(1) The canonical projection $N\otimes_{A^e} M \to N\otimes_{\Gamma} M$ induces $(N\otimes_{A^e} M)_H\to N\otimes_{\Gamma} M$ by \eqref{eq:galois-1}, which is surjective. In order to prove (1), it is sufficient to show that for any $\sum_p x_p\otimes y_p\in \Gamma$,
	\[
	\sum_p y_pn x_p\otimes_{A^e} m-\sum_p n\otimes_{A^e} x_pmy_p\in [H, N\otimes_{A^e}M],
	\]
	or equivalently
	\[
	\sum_p y_pn x_p\otimes_{A^e} m=\sum_p n\otimes_{A^e} x_pmy_p
	\]
	as elements of $(N\otimes_{A^e} M)_H$.
	
	Let
	\[
	z=\sum_{p,i,j} x_{p[0]}l_i(x_{p[1]})\otimes_A r_i(x_{p[1]})\otimes l_j(x_{p[2]})\otimes_A r_j(x_{p[2]})y_p.
	\]
	We then have
	\begin{align*}
	(I\otimes_A I &\otimes \gamma)(z)=\sum_{p,i,j} x_{p[0]}l_i(x_{p[1]})\otimes_A r_i(x_{p[1]})\otimes \gamma(l_j(x_{p[2]})\otimes_A r_j(x_{p[2]})y_{p}) \\
	&= \sum_{p,i,j} x_{p[0]}l_i(x_{p[1]})\otimes_A r_i(x_{p[1]})\otimes \gamma(l_j(x_{p[2]})\otimes_A r_j(x_{p[2]}))\gamma(1\otimes_A y_{p}) \\
	&= \sum_{p,i,j} x_{p[0]}l_i(x_{p[1]})\otimes_A r_i(x_{p[1]})\otimes l_j(x_{p[2]})\otimes_A r_j(x_{p[2]})y_{p[0]} \otimes x_{p[3]}y_{p[1]} \\
	&= \sum_{p,i,j} x_{p[0]}l_i(x_{p[1]})\otimes_A r_i(x_{p[1]})\otimes l_j(x_{p[2]})\otimes_A r_j(x_{p[2]})y_{p} \otimes x_{p[3]}Sx_{p[4]} \\
	&= \sum_{p,i,j} x_{p[0]}l_i(x_{p[1]})\otimes_A r_i(x_{p[1]})\otimes l_j(x_{p[2]})\otimes_A r_j(x_{p[2]})y_{p} \otimes \eps(x_{p[3]}) \\
	&=z\otimes 1,
	\end{align*}
	yielding $z\in B\otimes_A B\otimes B\otimes_A A$ by \eqref{eq:cotensor-1}. Notice that $z$ is also equal to
	\[
	\sum_{p,i,j} x_p l_i(S^{-1}y_{p[2]})\otimes_A r_i(S^{-1}y_{p[2]})\otimes l_j(S^{-1}y_{p[1]})\otimes_A r_j(S^{-1}y_{p[1]})y_{p[0]},
	\]
	and apply $\gamma'\otimes I\otimes_A I$ to $z$ which is expressed as above. Using a similar argument, we obtain $z\in A\otimes_A B\otimes B\otimes_A B$ by \eqref{eq:cotensor-2}, and thus $z\in A\otimes_A B\otimes B\otimes_A A$. Applying Lemma \ref{lem:h-gamma} or \eqref{eq:galois-1}, we have $z\in A\otimes_A \Gamma\otimes_A A$.
	
	Accordingly, in $(N\otimes_{A^e} M)_H$,  we obtain
	\begin{align*}
	\sum_p y_pn x_p\otimes_{A^e} m&=\sum_p \eps(x_{p[1]})y_pn x_{p[0]}\otimes_{A^e} m \\
	&=\sum_p x_{p[1]}\rightharpoonup(y_pn x_{p[0]}\otimes_{A^e} m) \\
	&=\sum_{p,i,j}  r_j(x_{p[2]})y_pn x_{p[0]}l_i(x_{p[1]}) \otimes_{A^e} r_i(x_{p[1]})ml_j(x_{p[2]}) \\
	&=\sum_{p,i,j}  n  \otimes_{A^e} x_{p[0]}l_i(x_{p[1]})r_i(x_{p[1]})ml_j(x_{p[2]})r_j(x_{p[2]})y_p \\
	&=\sum_{p}  n  \otimes_{A^e} x_{p[0]}\eps(x_{p[1]})m\eps(x_{p[2]})y_p \\
	&=\sum_{p}  n  \otimes_{A^e} x_{p}my_p,
	\end{align*}
	so the proof of the first assertion is finished.
	
	(2) The inclusion $\Hom_\Gamma(M, N)\subseteq \Hom_{A^e}(M, N)^H$ is obvious, by \eqref{eq:galois-1}. For the opposite direction, the argument is quite similar to the above. For any $f\in \Hom_{A^e}(M, N)^H$, we have $(f\leftharpoonup h)(m)=\eps(h)m$. In particular,
	\begin{align*}
	\sum_px_pf(m)y_p&=\sum_px_{p[0]}\eps(x_{p[1]})f(m)y_p=\sum_px_{p[0]}(f\leftharpoonup x_{p[1]})(m)y_p \\
	&=\sum_{p,i,j} x_{p[0]}l_i(x_{p[1]}) f(r_i(x_{p[1]})m l_j(x_{p[2]})) r_j(x_{p[2]})y_p \\
	&=\sum_{p,i,j}  f(x_{p[0]}l_i(x_{p[1]})r_i(x_{p[1]})m l_j(x_{p[2]})r_j(x_{p[2]})y_p ) \\
	&=\sum_{p}  f(x_{p[0]}\eps(x_{p[1]}) m \eps(x_{p[2]})y_p ) =f\Biggl(\sum_p x_pmy_p\biggr).
	\end{align*}
	Therefore, $f$ is $\Gamma$-linear, as desired.
\end{proof}

\begin{lem}\label{lem:flat-over-d}
	Let $B/A$ be a flat $H$-Galois extension, and $\Gamma$ be as above. Then
	\begin{enumerate}
		\item $A$ is a left $\Gamma$-module in a natural way,
		\item $B^e$ is flat as a right $\Gamma$-module, and $B^e\otimes_\Gamma A\cong B$ via $(b\otimes b')\otimes_\Gamma a\mapsto bab'$.
	\end{enumerate}
\end{lem}

\begin{proof}
	(1) was proved in \cite[Lemma 2.2]{Wang18}, and (2) follows from the isomorphism $B\otimes_A \Gamma\cong B^e$, which is directly deduced from the proof of \cite[Proposion 2.3]{CCMT07}.
\end{proof}

Now we can reprove Stefan's spectral sequences.

\begin{thm}[\cite{Ste95}]
	Let $B/A$ be a flat $H$-Galois extension. Then for any $B$-bimodule $N$, there are convergent spectral sequences
	\begin{gather*}
	E^2_{pq}=\upH_p(H, \upH_q(A, N))\Longrightarrow \upH_{p+q}(B, N), \\
	E_2^{pq}=\upH^p(H, \upH^q(A, N))\Longrightarrow \upH^{p+q}(B, N).
	\end{gather*}
\end{thm}

\begin{proof}
	Let us prove the cohomology spectral sequence, and the homology version is omitted since the argument is similar.
	
	Choose a projective resolution $P$ of the right $H$-module $\Bbbk$, and an injective resolution $I$ of the $B^e$-module $N$. By Lemma \ref{lem:flat-over-d} (1), $A$ is a left $\Gamma$-module, and thus $\Hom_{A^e}(A, I)$ admits a right $H$-module structure, as stated in Proposition \ref{prop:module-structure}.  Consider the double complex $C^{\bullet\bullet}$ with $C^{pq}=\Hom_{H}(P_p, \Hom_{A^e}(A, I^q))$.
	
	Since $B$ is flat over $A$, $I^q$ viewed as an $A^e$-module is injective. So the first and second pages of the spectral sequence ${}^\mathrm{I}E^{\bullet\bullet}$ induced by column filtration of $C^{\bullet\bullet}$ are given by
	\[
	{}^\mathrm{I} E_1^{pq}=\Hom_{H}(P_p, \upH^q(A, N))\quad\text{and}\quad {}^\mathrm{I}E_2^{pq}=\mathrm{Ext}^p_{H}(\Bbbk, \upH^q(A, N))
	\]
	respectively. In order to compute the spectral sequence ${}^\mathrm{II}E^{\bullet\bullet}$ induced by row filtration, we employ the technical result that $\Hom_{A^e}(A, I^p)$ is $\Hom_H(\Bbbk, -)$-acyclic for all injective $B^e$-module $I^p$, proved in \cite[Proposition 3.2]{Ste95}. As a consequence,
	\[
	{}^\mathrm{II} E_1^{pq}=\mathrm{Ext}^q_{H}(\Bbbk, \Hom_{A^e}(A, I^p))=
	\begin{cases}
	\Hom_{H}(\Bbbk, \Hom_{A^e}(A, I^p)), & q=0, \\ 0, & q\neq 0.
	\end{cases}
	\]
	On the other hand, we have
	\begin{align*}
	\Hom_{H}(\Bbbk, \Hom_{A^e}(A, I^p))&=\Hom_{A^e}(A, I^p)^H=\Hom_{\Gamma}(A, I^p) \\
	&\cong \Hom_{B^e}(B^e\otimes_\Gamma A, I^p)\cong \Hom_{B^e}(B, I^p),
	\end{align*}
	by Proposition \ref{prop:invariant} and Lemma \ref{lem:flat-over-d} (2). Taking cohomology, we obtain
	\[
	{}^\mathrm{II} E_2^{pq}=
	\begin{cases}
	\upH^p(B, N), & q=0, \\ 0, & q\neq 0,
	\end{cases}
	\]
	which collapses on this page.
	
	Let us equip $\upH^q(A, N)$ with the trivial left $H$-module structure, and then
	\[
	{}^\mathrm{I}E_2^{pq}=\mathrm{Ext}^p_{H}(\Bbbk, \upH^q(A, N))\cong \upH^p(H, \upH^q(A, N)).
	\]
	Therefore, $\upH^p(H, \upH^q(A, N))\Longrightarrow \upH^{p+q}(B, N)$ is obtained.
\end{proof}

\section{Cup product of Hochschild cohomology in semisimple case}\label{sec:cup-semisimple}

In this section, let $B/A$ be a $H$-Galois extension where $H$ is semisimple. By \cite{KT81}, $B$ is finitely generated projective over $A$, and hence $B/A$ is a flat extension. In this case, Stefan's (cohomological version) spectral sequence yields a nice isomorphism $\upH^\bullet(B, N)\cong \upH^\bullet(A, N)^H$. The isomorphism will be constructed explicitly using an appropriate resolution.

\begin{lem}\label{lem:bialgeboid}
	Let $K$, $N$ be two left $\Gamma$-modules. Then $K\otimes_A N$ is also a left $\Gamma$-module.
\end{lem}

\begin{proof}
	Let $\sum_p x_p\otimes y_p\in \Gamma$. We have
	\[
	\mathfrak{r}:=\sum_{i,p}x_{p[0]}\otimes l_i(x_{p[1]})\otimes_A r_i(x_{p[1]})\otimes y_p \in B\otimes B\otimes_A B\otimes B^{\op}.
	\]
	Now regard the second tensor summand of $\mathfrak{r}$ as elements of $B^{\op}$, and thus $\mathfrak{r}\in B^e\otimes_A B^e$. Observe that by \eqref{eq:galois-2}, \eqref{eq:galois-3}, and the fact
	\[
	\mathfrak{r}=\sum_{j,p}x_{p}\otimes l_j(S^{-1}y_{p[1]})\otimes_A r_j(S^{-1}y_{p[1]})\otimes y_{p[0]},
	\]
	we have
	\begin{gather*}
	(I\otimes \gamma\otimes I)(\mathfrak{r})=(I\otimes I\otimes_A I\otimes \rho')(\mathfrak{r}), \\
	(\rho\otimes I\otimes_A I\otimes I)(\mathfrak{r})=(I\otimes \gamma'\otimes I)(\mathfrak{r}).
	\end{gather*}
	It follows that
	\begin{equation}\label{eq:frak-r}
	\mathfrak{r}\in (B^e\otimes_A \Gamma)\cap (\Gamma\otimes_A B^e)=\Gamma\otimes_A \Gamma.
	\end{equation}
	We define $K\times N\to K\otimes_A N$ by $(k, n)\mapsto \sum_{i,p}x_{p[0]}k l_i(x_{p[1]})\otimes_A r_i(x_{p[1]})n y_p$, which is $A$-balanced by \eqref{eq:galois-5}, so it induces a map $K\otimes_A N\to K\otimes_A N$ given by
	\[
	k\otimes_A n\mapsto \sum_{i,p}x_{p[0]}k l_i(x_{p[1]})\otimes_A r_i(x_{p[1]})n y_p.
	\]
	Furthermore, this gives rise to a $\Gamma$-module structure on $K\otimes_AN$ by \eqref{eq:galois-4}.
\end{proof}

Notice that $A$ is a $\Gamma$-module by Lemma \ref{lem:flat-over-d}. The canonical map $A\otimes_AA\cong A$ is an isomorphism of $\Gamma$-modules, not only of $A$-bimodules.

In order to investigate the formula of cup product, we need two hypotheses.

\begin{hypo}
There exists a left $\Gamma$-module resolution $K$ of $A$ such that $K_m$ is projective as an $A^e$-module for all $m$.
\end{hypo}

\begin{hypo}
	There exists a morphism $T\colon K\to K\otimes_AK$ of left $\Gamma$-module complexes which lifts the isomorphism $A\cong A\otimes_AA$.
\end{hypo}

\begin{thm}\label{thm:isomorphism-semisimple}
	Let $B/A$ be an $H$-Galois extension where $H$ is semisimple. Suppose that Hypothesis I holds true. Then for any $B$-bimodule $N$, there is an isomorphism $\upH^\bullet(A, N)^H\cong \upH^\bullet(B, N)$ induced by the isomorphism
	\[
	\Theta\colon\Hom_{A^e}(K, N)^H\To \Hom_{B^e}(B^e\otimes_\Gamma K, N)
	\]
	of complexes given by $\Theta(f)(b\otimes b'\otimes_\Gamma k)=bf(k)b'$. In particular, $\upH^\bullet(A, B)^H\cong \upHH^\bullet(B)$.
\end{thm}

\begin{proof}
	Obviously, $\Theta$ is composed as follows
	\[
	\Hom_{A^e}(K, N)^H= \Hom_{\Gamma}(K, N)\cong \Hom_{B^e}(B^e\otimes_\Gamma K, N)
	\]
	by Proposition \ref{prop:invariant}. Since $H$ is semisimple, taking $H$-invariants is an exact functor. Hence for all $m$,
	\[
	\upH^m(\Hom_{A^e}(K, N)^H)\cong \upH^m(\Hom_{A^e}(K, N))^H=\upH^m(A, N)^H.
	\]
	It suffices to show that $B^e\otimes_\Gamma K$ is a projective resolution of $B$ over $B^e$. To this end, we see that $B^e\otimes_\Gamma K\to B\to 0$ is exact by Lemma \ref{lem:flat-over-d} (2). Furthermore, $\Hom_{B^e}(B^e\otimes_\Gamma K_m, -)$ is naturally isomorphic to $(-)^H\circ \Hom_{A^e}(K_m, -)$ by Proposition \ref{prop:invariant}, and the latter is exact. Thus $B^e\otimes_\Gamma K_m$ is a projective $B^e$-module.
\end{proof}

In the case $N=B$, the isomorphism in Theorem \ref{thm:isomorphism-semisimple} is $\upH^\bullet(A, B)^H\cong \upHH^\bullet(B)$. Notice that there are cup products on both $\upHH^\bullet(B)$ and $\upH^\bullet(A, B)$. Furthermore, the cup products can be computed via an arbitrary projective resolution (see Sec.~\ref{sec:preliminaries}). We plan to show that $\Theta^{-1}$ preserves the cup products at the level of complex, and consequently, $\upH^\bullet(A, B)^H$ has a cup product structure inherited from $\upH^\bullet(A, B)$, which coincides with the one of $\upHH^\bullet(B)$.

Denote $\KK=B^e\otimes_\Gamma K$. Under Hypotheses I and II, we can construct a morphism $\TT\colon \KK\to \KK\otimes_B\KK$ of $B$-bimodule complexes which lifts the canonical isomorphism  $B\cong B\otimes_BB$. Let us introduce the construction.

First of all, since $T$ is a morphism of complexes, we have $T\circ d_K=d_{K\otimes K}\circ T$. If we write $T(k)=\sum k_{\langle 1\rangle}\otimes_A k_{\langle 2\rangle}$ for any chain $k$ of $K$, then
\[
\sum d_K(k)_{\langle 1\rangle}\otimes_A d_K(k)_{\langle 2\rangle}=\sum d_K(k_{\langle 1\rangle})\otimes_A k_{\langle 2\rangle}+(-1)^{|k_{\langle 1\rangle}|}k_{\langle 1\rangle}\otimes_A d_K(k_{\langle 2\rangle}).
\]
For simplicity, the symbol $\sum$  is often suppressed if no confusion arises. Moreover, since $T$ preserves the $\Gamma$-action, we have
\[
\sum_p (x_pky_p)_{\langle 1\rangle}\otimes_A (x_pky_p)_{\langle 2\rangle}=\sum_{i,p} x_{p[0]}k_{\langle 1\rangle}l_i(x_{p[1]})\otimes_A r_i(x_{p[1]})k_{\langle 2\rangle}y_p
\]
for all $\sum_px_p\otimes y_p\in \Gamma$, by (the proof of) Lemma \ref{lem:bialgeboid}.

Secondly, for each $m\geq 0$, we define a map $\widetilde{\TT}_m\colon B^e\times K_m\to (\KK\otimes_B\KK)_m$ by
\[
\widetilde{\TT}_m(b\otimes b', k)=(b\otimes 1\otimes_\Gamma k_{\langle 1\rangle})\otimes_B (1\otimes b'\otimes_\Gamma k_{\langle 2\rangle}).
\]
In order to obtain $\TT_m\colon  \KK_m\to (\KK\otimes_B\KK)_m$, it suffices to show that $\widetilde{\TT}_m$ is $\Gamma$-balanced. In fact, for any $\sum_p x_p\otimes y_p\in \Gamma$, we have
\begin{align*}
\widetilde{\TT}_m\biggl((b\otimes b')\sum_p x_p\otimes y_p, k\biggr)&=\sum_p\widetilde{\TT}_m(bx_p\otimes y_pb', k)\\
&=\sum_p(bx_p\otimes 1\otimes_\Gamma k_{\langle 1\rangle})\otimes_B(1\otimes y_p b'\otimes_\Gamma k_{\langle 2\rangle}),
\end{align*}
and
\begin{align*}
\widetilde{\TT}_m\biggl(b\otimes b', &\sum_p x_pk y_p\biggr)=\sum_p(b\otimes 1\otimes_\Gamma (x_pk y_p)_{\langle 1\rangle})\otimes_B(1\otimes b'\otimes_\Gamma(x_pk y_p)_{\langle 2\rangle})\\
&=\sum_{i,p} (b\otimes 1\otimes_\Gamma x_{p[0]}k_{\langle 1\rangle}l_i(x_{p[1]}))\otimes_B (1\otimes b'\otimes_\Gamma r_i(x_{p[1]})k_{\langle 2\rangle}y_p) \\
&\overset{\dagger}{=}\sum_{i,p} (bx_{p[0]}\otimes l_i(x_{p[1]})\otimes_\Gamma k_{\langle 1\rangle})\otimes_B (r_i(x_{p[1]})\otimes y_pb'\otimes_\Gamma k_{\langle 2\rangle}) \\
&=\sum_{i,p} (bx_{p[0]}\otimes l_i(x_{p[1]})r_i(x_{p[1]})\otimes_\Gamma k_{\langle 1\rangle})\otimes_B ( 1\otimes y_pb'\otimes_\Gamma k_{\langle 2\rangle}) \\
&=\sum_{p} (bx_{p[0]}\otimes \eps(x_{p[1]})\otimes_\Gamma k_{\langle 1\rangle})\otimes_B ( 1\otimes y_pb'\otimes_\Gamma k_{\langle 2\rangle}) \\
&=\sum_{p} (bx_{p}\otimes 1\otimes_\Gamma k_{\langle 1\rangle})\otimes_B ( 1\otimes y_pb'\otimes_\Gamma k_{\langle 2\rangle}),
\end{align*}
where the equation marked by $\dagger$ holds true by \eqref{eq:frak-r}. So $\widetilde{\TT}_m$ induces a homomorphism $\TT_m\colon  \KK_m\to (\KK\otimes_B\KK)_m$ of $B$-bimodules.

Next, let us check that all $\TT_m$'s constitute a morphism $\TT\colon \KK\to \KK\otimes_B\KK$ of $B$-bimodule complexes. To see this, we have to prove that $\TT$ commutes with the differentials. In fact, 
\begin{align*}
(\TT\circ{} & d_{\KK} )(b\otimes b'\otimes_\Gamma k)=\TT(b\otimes b'\otimes_\Gamma d_K(k)) \\
&=(b\otimes 1\otimes_\Gamma d_K(k)_{\langle 1\rangle})\otimes_B (1\otimes b'\otimes_\Gamma d_K(k)_{\langle 2\rangle}) \\
&=(b\otimes 1\otimes_\Gamma d_K(k_{\langle 1\rangle}))\otimes_B (1\otimes b'\otimes_\Gamma k_{\langle 2\rangle}) \\
&\mathrel{\phantom{=}}{}+(-1)^{|k_{\langle 1\rangle}|}(b\otimes 1\otimes_\Gamma k_{\langle 1\rangle})\otimes_B (1\otimes b'\otimes_\Gamma d_K(k_{\langle 2\rangle})) \\
&=d_{\KK}(b\otimes 1\otimes_\Gamma k_{\langle 1\rangle})\otimes_B (1\otimes b'\otimes_\Gamma k_{\langle 2\rangle}) \\
&\mathrel{\phantom{=}}{}+(-1)^{|b\otimes 1\otimes_\Gamma k_{\langle 1\rangle}|}(b\otimes 1\otimes_\Gamma k_{\langle 1\rangle})\otimes_B d_{\KK}(1\otimes b'\otimes_\Gamma k_{\langle 2\rangle}) \\
&=d_{\KK\otimes_B\KK}\bigl( (b\otimes 1\otimes_\Gamma k_{\langle 1\rangle})\otimes_B (1\otimes b'\otimes_\Gamma k_{\langle 2\rangle}) \bigr) \\
&=(d_{\KK\otimes_B\KK}\circ \TT)(b\otimes b'\otimes_\Gamma k).
\end{align*}

The last step is to prove that $\TT$ lifts to the canonical isomorphism $B\cong B\otimes_BB$. This is similar to the proof of $\TT\circ d_{\KK} =d_{\KK\otimes_B\KK}\circ \TT$, so we omit it.

\begin{thm}\label{thm:cup-product-semisimple}
	Let $B/A$ be an $H$-Galois extension where $H$ is semisimple. Suppose that Hypotheses I and II hold true. Then  the isomorphism $\upH^\bullet(A, B)^H\cong \upHH^\bullet(B)$ in Theorem \ref{thm:isomorphism-semisimple} preserves cup products.
\end{thm}

\begin{proof}
	Let $f$, $g$ be two cochains of $\Hom_{A^e}(K, B)$. Recall that the cup product $f\smallsmile g$ is defined as $\mu\circ(f\otimes_A g)\circ T$ where $\mu$ is the multiplication of $B$. On elements, we have $(f\smallsmile g)(k)=f(k_{\langle 1\rangle})g(k_{\langle 2\rangle})$.  Similarly the cup product $\tilde{f}\smallsmile\tilde{g}$ of two cochains $\tilde{f}$, $\tilde{g}$ of $\Hom_{B^e}(\KK, B)$ is defined as $\mu\circ(\tilde{f}\otimes_B\tilde{g})\circ\TT$.
	
	Notice that we have an isomorphism $\Theta^{-1}\colon \Hom_{B^e}(\KK, B)\to \Hom_{A^e}(K, B)^H$ given by $\Theta^{-1}(\tilde{f})(k)=\tilde{f}(1\otimes 1\otimes_\Gamma k)$. Let $\Psi$ be the composition of $\Theta^{-1}$ and the embedding $\Hom_{A^e}(K, B)^H\hookrightarrow \Hom_{A^e}(K, B)$. We claim that $\Psi$ preserves cup products. 	In fact, $\Psi(\tilde{f}\smallsmile\tilde{g})(k)=(\tilde{f}\smallsmile\tilde{g})(1\otimes 1\otimes_\Gamma k)=\tilde{f}(1\otimes 1\otimes_\Gamma k_{\langle 1\rangle})\tilde{g}(1\otimes 1\otimes_\Gamma k_{\langle 2\rangle})=\Psi(\tilde{f})(k_{\langle 1\rangle})\Psi(\tilde{g})(k_{\langle 2\rangle})=(\Psi(\tilde{f})\smallsmile\Psi(\tilde{g}))(k)$.
	
	It follows that $\Hom_{A^e}(K, B)^H$ is closed under taking cup product. Descending to cohomology, the isomorphism $\upH^\bullet(A, B)^H\cong \upHH^\bullet(B)$ preserves cup products, as desired.
\end{proof}

\begin{rem}
When $\Gamma$ is a projective $A^e$-module, Hypotheses I and II are always satisfied because $K$ can be chosen as any projective resolution of $A$ over $\Gamma$.
\end{rem}

\section{Hochschild cohomology of smash products over Koszul algebras}\label{sec:hoch-smash-koszul}

In this section, let us consider the special situation that $A$ is a Koszul algebra and $B=A\# H$ where $H$ is a finite dimensional semisimple Hopf algebra, acting on $A$ homogeneously. In this case, $\sum_il_i(h)\otimes_A r_i(h)=\sum(1\# Sh_{(1)})\otimes_A(1\# h_{(2)})$, and $\Gamma$ is spanned by $\sum(r\# h_{(1)})\otimes ((Sh_{(3)}\triangleright r')\# Sh_{(2)})$ for all $r$, $r'\in A$ and $h\in H$. For simplicity, we omit $\sum$ and $\#$ in the above expressions. Thus $\sum(r\# h_{(1)})\otimes ((Sh_{(3)}\triangleright r')\# Sh_{(2)})$ will be written as $rh_{(1)}\otimes Sh_{(2)}r'$.

Hypotheses I and II are always true for smash products since we can choose $K$ to be the bar resolution $B(A)$ and $T$ to be the morphism induced by the comultiplication of $\mathsf{T}^c(A)$ (cf.~\cite{Vol}). However, the bar resolution is too big, and the Koszul complex will be a good alternative to help us compute cohomological groups. Hence, our main aim in this section is to verify that Hypotheses I and II are fulfilled for any Koszul algebra $A$ and the Koszul complex $K(A)$.

Since $H$ acts on $A$ homogeneously, each $h\in H$ can be regarded as a linear map on $A_1=V$, sending $v\in V$ to $h\triangleright v$. Let $h$ act on $V^{\otimes m}$ diagonally, namely, $h\triangleright(v_1\otimes \cdots\otimes v_m)= h_{(1)}\triangleright v_1\otimes \cdots\otimes h_{(m)}\triangleright v_m$, so $V^{\otimes m}$ becomes an $H$-module. Since $A$ is an $H$-module algebra, we have $h\triangleright R\subseteq R$ and hence $(A^!_m)^*$ is an $H$-submodule of $V^{\otimes m}$, by \eqref{eq:koszul-embedding}. Define $\Gamma\times K(A)_m\to K(A)_m$ by
\begin{equation}\label{eq:koszul-gamma-module}
(rh_{(1)}\otimes Sh_{(2)}r', a\otimes \alpha\otimes a')\mapsto r(h_{(1)}\triangleright a)\otimes h_{(2)}\triangleright\alpha\otimes (h_{(3)}\triangleright a')r'.
\end{equation}
This makes $K(A)_m$ into a left $\Gamma$-module. Furthermore, it follows from \cite[Proposition 2.1]{LWZ12} that $K(A)$ is a left $\Gamma$-module complex, that is, Hypothesis I is true.

As we mentioned in the beginning of this section, both hypotheses are true for the bar complex $B(A)$ and the morphism $T\colon B(A)\to B(A)\otimes_A B(A)$ defined as
\[
T(a_0\otimes a_1\otimes \cdots\otimes a_{m+1})=\sum_{u=0}^{m}(a_0\otimes \cdots\otimes a_{u}\otimes 1)\otimes_A(1\otimes a_{u+1}\otimes \cdots\otimes a_{m+1}).
\]
By \eqref{eq:koszul-embedding}, we obtain a well-defined morphism $K(A)\to K(A)\otimes_A K(A)$ by restricting $T$ to $K(A)$. The morphism is also denoted by $T$, by abuse of notations. As a result, Hypothesis II is true. Furthermore, $T$ satisfies the cocommutative law in the sense of $(T\otimes_A I)\circ T=(I\otimes_A T)\circ T\colon K(A)\to K(A)\otimes_A K(A)\otimes_A K(A)$. It follows that $\Hom_{A^e}(K(A), B)$ equipped with the cup product is a differential graded algebra.

\begin{rem}
	Without the Koszul assumption, $\Hom_{A^e}(K, B)$ fails to be a differential graded algebra, even if Hypotheses I and II are fulfilled. Instead, $\Hom_{A^e}(K, B)$ is an $\mathrm{A}_\infty$-algebra.
\end{rem}

We state the following proposition, which is easily obtained from Theorem \ref{thm:cup-product-semisimple}.

\begin{prop}\label{prop:koszul-cup-1}
	Let $A$ be a Koszul algebra, and $H$ be a finite dimensional semisimple Hopf algebra. Suppose that $H$ acts on $A$ homogeneously, and that $A$ is an $H$-module algebra. Then  the isomorphism $\upH^\bullet(A, A\#H)^H\cong \upHH^\bullet(A\#H)$ in Theorem \ref{thm:isomorphism-semisimple} preserves cup products.
\end{prop}

In order to compute the cup product on $\upHH^\bullet(A\#H)$ explicitly, we have to determine the cup product on $\Hom_{A^e}(K(A), A\# H)$. Notice that the composition
\[
\Phi^m\colon A_m^!\otimes (A\# H)  \cong \Hom((A^!_m)^*, A\# H)\cong \Hom_{A^e}(K(A)_m, A\# H)
\]
is an isomorphism of vector spaces, in which $\Phi^m( \xi\otimes b)\colon K(A)_m\to A\# H$ is given by for any $\xi\in A^!$ and $b\in A\#H$,
\[
\Phi^m(\xi\otimes b)(a\otimes\alpha\otimes a')=\alpha(\xi)aba'.
\]
Let us write the $m$-th differential $(d_K^*)^m$ of $\Hom_{A^e}(K(A), A\# H)$ as $d^m$ for short. All $\Phi^m$'s constitute an isomorphism $A^!\otimes (A\# H) \cong \Hom_{A^e}(K(A), A\# H)$ of complexes, if the differential $\partial$ of $A^!\otimes (A\# H)$ is defined by $\partial^m=(\Phi^{m+1})^{-1}\circ d^m\circ \Phi^m$.

Suppose $\dim A^!_m=\kappa_m$. Let $\alpha^{m1},\ldots, \alpha^{m\kappa_m}$ be a basis of $A^!_m$, and $\alpha_{m1},\ldots, \alpha_{m\kappa_m}$ be the dual basis of $(A^!_m)^*$. 
Thus, $(\Phi^m)^{-1}$ is given by
\[
(\Phi^m)^{-1}(f)=\sum_{j=1}^{\kappa_m} \alpha^{mj}\otimes f(1\otimes \alpha_{mj}\otimes 1).
\]
Consequently,
\begin{align*}
\partial^m&(\xi\otimes b)=(\Phi^{m+1})^{-1}(d^m\circ \Phi^m(\xi\otimes b) ) \\
&=\sum_{j=1}^{\kappa_{m+1}} \alpha^{m+1,j}\otimes d^m\circ \Phi^m(\xi\otimes b)(1\otimes \alpha_{m+1,j}\otimes 1) \\
&=\sum_{j=1}^{\kappa_{m+1}} \alpha^{m+1,j}\otimes \Phi^m(\xi\otimes b)(d_K^{m+1}(1\otimes \alpha_{m+1,j}\otimes 1))  \\
&=\sum_{j=1}^{\kappa_{m+1}}\sum_{i=1}^n \bigl(\alpha^{m+1,j} \otimes \Phi^m(\xi\otimes b)(e_i\otimes \alpha_{m+1,j}e^i\otimes 1) \\
&\mathrel{\phantom{=}} {}-(-1)^m \alpha^{m+1,j}\otimes \Phi^m(\xi\otimes b)(1\otimes e^i\alpha_{m+1,j}\otimes e_i) \bigr) \\
&=\sum_{j=1}^{\kappa_{m+1}}\sum_{i=1}^n \bigl(\alpha^{m+1,j} \otimes (\alpha_{m+1,j}e^i)(\xi)e_ib
- (-1)^m \alpha^{m+1,j} \otimes (e^i\alpha_{m+1,j})(\xi)be_i \bigr) \\
&=\sum_{j=1}^{\kappa_{m+1}}\sum_{i=1}^n \bigl(\alpha^{m+1,j} \otimes\alpha_{m+1,j}(e^i\xi)e_ib- (-1)^m \alpha^{m+1,j} \otimes\alpha_{m+1,j}(\xi e^i)be_i \bigr) \\
&=\sum_{i=1}^n \sum_{j=1}^{\kappa_{m+1}} \bigl( \alpha_{m+1,j}(e^i\xi)\alpha^{m+1,j} \otimes e_ib - (-1)^m \alpha_{m+1,j}(\xi e^i)\alpha^{m+1,j}\otimes be_i  \bigr) \\
&=\sum_{i=1}^n \bigl(e^i\xi\otimes e_ib - (-1)^m \xi e^i\otimes be_i \bigr).
\end{align*}

We have concluded that $A^!\otimes (A\# H)\cong\Hom_{A^e}(K(A), A\# H)$ as complexes. Furthermore, we will show that they are isomorphic as differential graded algebras. Let us check that $A^!\otimes (A\# H)$ is a differential graded algebra. For any elements $\xi\otimes b\in A^!_m\otimes (A\# H)$ and $\xi'\otimes b'\in A^!_{m'}\otimes (A\# H)$, we have
\begin{align*}
\partial((\xi &\otimes b)(\xi'\otimes b'))= \partial(\xi\xi'\otimes bb') \\
&=\sum_{i=1}^n \bigl(e^i\xi\xi'\otimes e_ibb' - (-1)^{m+m'} \xi\xi'e^i\otimes bb'e_i \bigr),
\end{align*}
and
\begin{align*}
\partial(\xi &\otimes b)(\xi'\otimes b')+(-1)^m(\xi \otimes b)\partial(\xi'\otimes b') \\
&=\sum_{i=1}^n \bigl(e^i\xi\otimes e_ib- (-1)^m \xi e^i\otimes be_i \bigr)(\xi'\otimes b') \\
&\mathrel{\phantom{=}} {}+(-1)^m(\xi \otimes b)\sum_{i=1}^n \bigl(e^i\xi'\otimes e_ib'- (-1)^{m'} \xi'e^i\otimes b'e_i \bigr) \\
&=\sum_{i=1}^n \bigl(e^i\xi\xi'\otimes e_ibb' - (-1)^m \xi e^i\xi'\otimes be_ib' \\
&\mathrel{\phantom{=}} {}+(-1)^m \xi e^i\xi'\otimes be_ib' - (-1)^{m+m'} \xi\xi'e^i\otimes bb'e_i \bigr) \\
&=\sum_{i=1}^n \bigl(e^i\xi\xi'\otimes e_ibb' - (-1)^{m+m'} \xi\xi'e^i\otimes bb'e_i \bigr).
\end{align*}
Therefore, $\partial((\xi \otimes b)(\xi'\otimes b'))=\partial(\xi \otimes b)(\xi'\otimes b')+(-1)^m(\xi \otimes b)\partial(\xi'\otimes b')$ holds true, forcing $A^!\otimes (A\#H)$ to be a differential graded algebra.

As normal, for any  differential graded algebra $\mathfrak{A}$, denote by $\upH^\bullet(\mathfrak{A})$ the cohomological algebra  of $\mathfrak{A}$.

\begin{thm}\label{thm:hochschild-cohomology-1}
	Let $A$ and $H$ be as above. Then $A^!\otimes (A\# H)\cong\Hom_{A^e}(K(A), A\# H)$ as differential graded algebras. As a consequence, $\upH^\bullet(A, A\# H)$ equipped with the cup product, as a graded algebra, is isomorphic to $\upH^\bullet(A^!\otimes (A\#H))$.
\end{thm}

\begin{proof}
	It is sufficient to prove that $\Phi\colon A^!\otimes (A\#H)\to\Hom_{A^e}(K(A), A\# H)$ is an algebra homomorphism, namely, to show $\Phi((\xi\otimes b)(\xi'\otimes b'))=\Phi(\xi\otimes b)\smallsmile \Phi(\xi'\otimes b')$ for all $\xi\otimes b\in A^!_m\otimes (A\# H)$ and $\xi'\otimes b'\in A^!_{m'}\otimes (A\# H)$.
	
	For any $a\otimes \alpha \otimes a'\in A\otimes (A^!_{m+m'})^*\otimes A$, we have $\Phi((\xi\otimes b)(\xi'\otimes b'))(a\otimes \alpha \otimes a')=\alpha(\xi\xi')abb'a'$. On the other hand, $T(a\otimes \alpha \otimes a')=(a\otimes \alpha_{\langle 1\rangle}\otimes 1)\otimes_A (1\otimes \alpha_{\langle 2\rangle}\otimes a')$, and thus
	\begin{align*}
	\bigl(\Phi(\xi\otimes b)&\smallsmile \Phi(\xi'\otimes b')\bigr) (a\otimes \alpha \otimes a')\\
	&=\Phi(\xi\otimes b)(a\otimes \alpha_{\langle 1\rangle}\otimes 1)\Phi(\xi'\otimes b')(1\otimes \alpha_{\langle 2\rangle}\otimes a') \\
	&=\alpha_{\langle 1\rangle}(\xi)ab \alpha_{\langle 2\rangle}(\xi')b'a' \\
	&= \alpha_{\langle 1\rangle}(\xi)\alpha_{\langle 2\rangle}(\xi') ab b'a' \\
	&=\alpha(\xi\xi')abb'a'.
	\end{align*}
	It follows that $\Phi$ is an isomorphism of graded algebras.
\end{proof}

Since $V$ is a left $H$-module, $V^*$ is a right $H$-module in a natural way, namely, $(\xi\triangleleft h)(v)=\xi(h\triangleright v)$ for all $\xi\in V^*$, $h\in H$ and $v\in V$. It follows that $(V^*)^{\otimes m}$ admits a right $H$-action diagonally, i.e.,
\[
(\xi_1\otimes \xi_2\otimes \cdots\otimes \xi_m)\triangleleft h=\xi_1\triangleleft h_{(1)}\otimes \xi_2\triangleleft h_{(2)}\otimes\cdots\otimes \xi_m\triangleleft h_{(m)}.
\]
It is easy to check that $(R^\perp)\triangleleft h\subseteq R^\perp$ for any $h\in H$, so $A^!$ is a right $H$-module algebra. Observe that $A\# H$ is naturally an $H$-bimodule, and hence $A^!\otimes (A\#H)$ becomes an $H$-$H\otimes H$-bimodule. We define the right $H$-action on $A^!\otimes (A\#H)$ by $(\xi\otimes b)\blacktriangleleft h=\xi\triangleleft h_{(2)}\otimes Sh_{(1)}bh_{(3)}$.

\begin{rem}
    Together with the right $H$-coaction
    \[
    A^!\otimes (A\#H)\to A^!\otimes (A\#H)\otimes H \qquad \xi\otimes ah\mapsto \xi\otimes ah_{(1)}\otimes h_{(2)},
    \]
    $A^!\otimes (A\#H)$ is in fact a Yetter--Drinfeld module.
\end{rem}

\begin{thm}\label{thm:cup-koszul}
	$\Phi\colon A^!\otimes (A\# H) \to\Hom_{A^e}(K(A), A\# H)$ is an isomorphism of complexes of right $H$-modules. Consequently, $\upHH^\bullet(A\# H)$ equipped with the cup product, is isomorphic to the $H$-invariant subalgebra of $\upH^\bullet(A^!\otimes (A\# H))$ as a graded algebra.
\end{thm}

\begin{proof}
	Due to Theorem \ref{thm:hochschild-cohomology-1}, it is sufficient to prove the first assertion.
	
	Recall that the right $H$-module structure on $\Hom_{A^e}(K(A)_m, A\# H)$ is defined in Proposition \ref{prop:module-structure}, and that $l_i(h)\otimes_A r_i(h)=Sh_{(1)}\otimes_A h_{(2)}$. So for any $f\in \Hom_{A^e}(K(A)_m, A\# H)$,
	\begin{align*}
	(f\leftharpoonup h)(a\otimes \alpha\otimes a')&=Sh_{(1)}f(h_{(2)}(a\otimes \alpha\otimes a')Sh_{(3)})h_{(4)} \\
	&=Sh_{(1)}f(h_{(2)}\triangleright a\otimes h_{(3)}\triangleright \alpha\otimes h_{(4)}\triangleright a')h_{(5)} \qquad \text{by \eqref{eq:koszul-gamma-module}} \\
	&=Sh_{(1)}(h_{(2)}\triangleright a)f(1\otimes h_{(3)}\triangleright\alpha\otimes 1)(h_{(4)}\triangleright a')h_{(5)} \\
	&=((Sh_{(2)}h_{(3)})\triangleright a)Sh_{(1)}f(1\otimes h_{(4)}\triangleright\alpha\otimes 1)(h_{(5)}\triangleright a')h_{(6)} \\
	&=aSh_{(1)}f(1\otimes h_{(2)}\triangleright \alpha\otimes 1)h_{(3)}a'.
	\end{align*}
	As a result,
	\begin{align*}
	(\Phi(\xi\otimes b)\leftharpoonup h)(a\otimes \alpha\otimes a')& =aSh_{(1)}\Phi(\xi\otimes b)(1\otimes h_{(2)}\triangleright\alpha\otimes 1)h_{(3)}a' \\
	&=(h_{(2)}\triangleright \alpha)(\xi) aSh_{(1)}bh_{(3)}a'\\
	&=\alpha(\xi\triangleleft h_{(2)})aSh_{(1)}bh_{(3)}a'.
	\end{align*}
	On the other hand,
	\begin{align*}
	\Phi((\xi\otimes b)\blacktriangleleft h)(a\otimes \alpha\otimes a')& =\Phi(\xi\triangleleft h_{(2)}\otimes Sh_{(1)}bh_{(3)})(a\otimes \alpha\otimes a') \\
	&=\alpha(\xi\triangleleft h_{(2)})aSh_{(1)}bh_{(3)}a'.
	\end{align*}
	It follows that $\Phi(\xi\otimes b)\leftharpoonup h=\Phi((\xi\otimes b)\blacktriangleleft h)$, i.e., $\Phi$ preserves the right $H$-action, as desired.
\end{proof}

\begin{rem}
    If $H=\Bbbk$ is the trivial Hopf algebra, then $A\# H\cong A$ and the previous theorem gives rise to an isomorphism $\upHH^\bullet(A)\cong\upH^\bullet(A^!\otimes A)$ of graded algebras, for all Koszul algebras $A$. This was originally proved in \cite{Neg17}.
\end{rem}

\section{Hochschild cohomology of quantum plane extended by Kac--Paljutkin Hopf algebra}\label{sec:kac}

The problem of classification of all Hopf algebras of low dimension was posed by Kaplansky in \cite{Kap75}. For example, classification of all types of Hopf algebras of dimension less than or equal to 11 over an algebraically closed field of characteristic 0 was presented in \cite{Ste99}. By \cite[Theorem 2.13]{Mas95}, there are 8 nonisomorphic classes for finite dimensional semisimple Hopf algebras of dimension 8 over a field of characteristic $\neq$ 2. Among them, there exists only one (up to isomorphism) semisimple and cosemisimple Hopf algebra of dimension 8; it is now famous as the Kac--Paljutkin Hopf algebra \cite{KP66,Mas08}.

Throughout this section, let $A= \Bbbk\langle u,v\rangle/(uv+vu)$ be the quantum $(-1)$-plane, and
$H$ be the Kac--Paljutkin Hopf algebra, which is generated as an algebra by $x$, $y$, $z$ with the relations
\[
x^{2}=y^{2}=1, z^{2}= \frac{1}{2}(1+x+y-xy), yx=xy, zx=yz, zy=xz.
\]
The coalgebra structure and the antipode of $H$ are given by
\begin{gather*}
\Delta(x)=x\otimes x, \Delta(y)=y\otimes y, \varepsilon(x)= \varepsilon(y) =1,\\
\Delta(z) = \frac{1}{2}(1\otimes 1 + 1\otimes x + y\otimes 1 - y\otimes x)(z\otimes z),  \varepsilon(z)=1,\\
S(x) = x, S(y) = y, S(z) = z.
\end{gather*}
There are several different ways that $A$ becomes an $H$-module algebra. We choose the following $H$-action on $A$:
\begin{alignat*}{2}
x\triangleright u&=u, &\quad x\triangleright v&=v, \\
y\triangleright u&=u, &\quad y\triangleright v&=v, \\
z\triangleright u&=q^{-1}v, &\quad z\triangleright v&=qu,
\end{alignat*}
where $q$ is an arbitrary nonzero scalar. The reader can check that under the action $A$ is an $H$-module algebra.

Notice that $H$ is $8$-dimensional with basis $\mathcal{B}=\{1,x,y,z,xy, xz, yz, xyz\}$. Taking the generating relations into account, $H$ is $\mathbb{Z}_2$-graded with $H_0$, $H_1$ spanned by $\{1,x,y,xy\}$, $\{z, xz,yz,  xyz\}$ respectively. Thus $A\# H$ is made into a $\mathbb{Z}_2$-graded algebra too. By a direct computation, the following equations
\begin{gather*}
h_0u=uh_0, h_0v=vh_0, \\
h_1u=q^{-1}vh_1, h_1v=quh_1
\end{gather*}
hold true in $A\# H$ for all $h_0\in H_0$, $h_1\in H_1$.

The Koszul dual of $A$ is $A^!=\Bbbk[u^*, v^*]/(u^{*2}, v^{*2})$. Note that $A^!_m=0$ for all $m\geq 3$, so it suffices to compute the zeroth, the first, and the second cohomological groups.

\subsection{Computation of $\upH^\bullet(A, A\# H)$}

In order to investigate $\upHH^\bullet(A\# H)$ as well as its cup product, let us begin with computing $\upH^\bullet(A, A\# H)$.

Recall that $\upH^\bullet(A, A\# H)$ is the cohomology of
\[
0\xrightarrow{\qquad} A^!_0\otimes(A\# H) \xrightarrow[\qquad]{\partial^0} A^!_1\otimes(A\# H) \xrightarrow[\qquad]{\partial^1} A^!_2\otimes(A\# H) \xrightarrow{\qquad} 0
\]
where the differential of $A^!\otimes (A\# H)$ is given by
\begin{align*}
\partial^0(1\otimes ah)&=u^*\otimes (uah- ahu)+v^*\otimes (vah- ahv), \\
\partial^1(u^*\otimes ah)&=u^*v^*\otimes (vah+ ahv), \\
\partial^1(v^*\otimes ah)&=u^*v^*\otimes (uah+ ahu).
\end{align*}
Obviously, the differentials preserve the $\mathbb{Z}_2$-grading of $A\#H$, and consequently all cohomological groups admit the induced $\mathbb{Z}_2$-grading, namely, $\upH^\bullet(A, A\#H)=\upH^\bullet(A, A\#H)_0\oplus \upH^\bullet(A, A\#H)_1$.

\begin{lem}\label{lem:h0-a}
	One has $\upH^0(A, A\#H)_0=\bigoplus_{i,j} (1\otimes u^{2i}v^{2j})H_0$ and $\upH^0(A, A\#H)_1=0$, where $(1\otimes u^{2i}v^{2j})H_0$ is the subspace of $A^!_0\otimes (A\# H)$ consisting of all $1\otimes u^{2i}v^{2j}h_0$ with $h_0\in H_0$.
\end{lem}

\begin{proof}
	It is well-known that for any $A$-bimodule $N$, $\upH^0(A, N)=\{n\in N\,|\, an=na\text{ for all generators }a \text{ of }A\}$. Thus, $\sum 1\otimes u^iv^jh^{ij}_0\in \upH^0(A, A\#H)_0$ if and only if
	\[
	\sum  uu^iv^jh^{ij}_0=\sum  u^iv^jh^{ij}_0u, \quad \sum  vu^iv^jh^{ij}_0=\sum  u^iv^jh^{ij}_0v,
	\]
	which is equivalent to
	\[
	\sum  u^{i+1}v^jh^{ij}_0=\sum  (-1)^ju^{i+1}v^jh^{ij}_0, \quad	\sum  (-1)^iu^iv^{j+1}h^{ij}_0=\sum  u^iv^{j+1}h^{ij}_0.
	\]
	It follows that $h^{ij}_0=0$ whenever $i$ or $j$ is odd. So
	\[
	\upH^0(A, A\#H)_0=\{1\otimes u^iv^jh^{ij}_0\,|\, i, j \text{ are even}\}=\bigoplus_{i,j}(1\otimes u^{2i}v^{2j})H_0.
	\]
	
	For $\upH^0(A, A\#H)_1$, the argument is similar. Hence $\sum 1\otimes u^iv^jh^{ij}_1\in \upH^0(A, A\#H)_1$ if and only if
	\[
	\sum u^{i+1}v^jh^{ij}_1=\sum q^{-1}u^iv^{j+1}h^{ij}_1, \quad	\sum (-1)^iu^iv^{j+1}h^{ij}_1 =\sum (-1)^jq u^{i+1}v^jh^{ij}_1.
	\]
	Assume that $\upH^0(A, A\#H)_1$ is nontrivial. By taking the lexicographic order on the set of $(i,j)$ with $h^{ij}_1\neq 0$ into account, we can easily deduce a contradiction. Therefore, $\upH^0(A, A\#H)_1=0$, as desired.
\end{proof}

\begin{lem}\label{lem:h1-a}
	$\upH^1(A, A\#H)_0=\bigoplus_{i,j} (u^*\otimes u^{2i+1}v^{2j})H_0\oplus\bigoplus_{i,j} (v^*\otimes u^{2i}v^{2j+1})H_0$, $\upH^1(A, A\#H)_1=\bigoplus_{j} (u^*\otimes v^{2j}-v^*\otimes qv^{2j})H_1$.
\end{lem}

\begin{proof}
	First of all, let us consider $\upH^1(A, A\#H)_0$. For any cohomological class in it, choose a representative
	\begin{equation}\label{eq:1-cocycle-01}
	\sum u^*\otimes u^iv^jh^{ij}_0+\sum v^*\otimes u^kv^l\tilde{h}^{kl}_0\in Z^1(A, A\#H)_0.
	\end{equation}
	Since
	\begin{equation}\label{eq:partial-0}
	\partial^0(1\otimes u^sv^th_0)=u^*\otimes u^{s+1}v^t(1-(-1)^t)h_0+v^*\otimes u^sv^{t+1}((-1)^s-1)h_0,
	\end{equation}
	we have
	\begin{gather*}
		u^*\otimes u^{i+1}v^{2j+1}h_0\equiv v^*\otimes u^iv^{2j+2}((-1)^i-1)\Bigl(-\frac{1}{2}h_0\Bigr) \pmod{\Ima\partial^0}, \\
	v^*\otimes u^{2k+1}v^{2l+1}h_0\equiv 0 \pmod{\Ima\partial^0}.
	\end{gather*}
	Thus, replacing the representative  if necessary, we may suppress the summands $u^*\otimes u^iv^jh^{ij}_0$ appearing in \eqref{eq:1-cocycle-01} with $i\geq 1$ and odd $j$, as well as $v^*\otimes u^kv^l\tilde{h}^{kl}_0$ with odd $k$, $l$. Hence \eqref{eq:1-cocycle-01} is uniquely rewritten as
	\begin{equation}\label{eq:1-cocycle-02}
	\sum u^*\otimes v^jh^{0j}_0+\sum u^*\otimes u^{i+1}v^{2j}h^{i+1,2j}_0+\sum v^*\otimes u^{2k}v^l\tilde{h}^{2k,l}_0+v^*\otimes u^{2k+1}v^{2l}\tilde{h}^{2k+1,2l}_0.
	\end{equation}
	After applying $\partial^1$ to \eqref{eq:1-cocycle-02}, we obtain
	\begin{align*}
	\sum 2v^{j+1}h^{0j}_0&+\sum((-1)^{i+1}+1)u^{i+1}v^{2j+1}h^{i+1,2j}_0\\
	&{}+\sum (1+(-1)^l)u^{2k+1}v^l\tilde{h}^{2k,l}_0+2u^{2k+2}v^{2l}\tilde{h}^{2k+1,2l}_0=0,
	\end{align*}
	which is simplified to
	\begin{align*}
	\sum 2v^{j+1}h^{0j}_0&+\sum2u^{2i+2}v^{2j+1}h^{2i+2,2j}_0\\
	&{}+\sum 2u^{2k+1}v^{2l}\tilde{h}^{2k,2l}_0+2u^{2k+2}v^{2l}\tilde{h}^{2k+1,2l}_0=0.
	\end{align*}
	It follows that $h^{0j}_0=h^{2i+2,2j}_0=\tilde{h}^{2k,2l}_0=\tilde{h}^{2k+1,2l}_0=0$. Then we drop the trivial summands of \eqref{eq:1-cocycle-02}, yielding
	\[
	\sum u^*\otimes u^{2i+1}v^{2j}h^{2i+1,2j}_0+\sum v^*\otimes u^{2k}v^{2l+1}\tilde{h}^{2k,2l+1}_0.
	\]
	So far, we have proved that each cohomological class in $\upH^1(A, A\#H)_0$ can be represented by a unique $1$-cocycle of the above form. Moreover, this cocycle does not belong to $\Ima\partial^0$, by \eqref{eq:partial-0}. As a result,
	\begin{align*}
	\upH^1(A, A\#H)_0&=\Bigl(\sum u^*\otimes u^{2i+1}v^{2j}h^{2i+1,2j}_0\Bigr)\oplus \Bigl(\sum v^*\otimes u^{2k}v^{2l+1}\tilde{h}^{2k,2l+1}_0\Bigl) \\
	&=\bigoplus_{i,j} (u^*\otimes u^{2i+1}v^{2j})H_0\oplus\bigoplus_{i,j} (v^*\otimes u^{2i}v^{2j+1})H_0.
	\end{align*}
	
	Next, let us compute $\upH^1(A, A\#H)_1$. Since
	\begin{align*}
	\partial^0(1\otimes u^sv^th_1)&=u^*\otimes (uu^{s}v^t h_1-u^sv^th_1u)+v^*\otimes (vu^{s}v^t h_1-u^sv^th_1v) \\
	&=u^*\otimes u^{s+1}v^th_1-u^*\otimes u^{s}v^{t+1}q^{-1}h_1+v^*\otimes (\cdots),
	\end{align*}
	we have
	\[
	u^*\otimes u^{i}v^jh_1\equiv u^*\otimes u^{i-1}v^{j+1}q^{-1}h_1+v^*\otimes (\cdots) \pmod {\Ima\partial^0}
	\]
	whenever $i\geq 1$. Itemizing the formula, we conclude
	\[
	u^*\otimes u^{i}v^jh_1\equiv u^*\otimes v^{i+j}q^{-i}h_1+v^*\otimes (\cdots) \pmod {\Ima\partial^0}.
	\]
	As what we dealt with $\upH^1(A, A\#H)_0$, there is a unique representative of the form
	\begin{equation}\label{eq:1-cocycle-11}
	\sum u^*\otimes v^{j}h^{0j}_1+\sum v^*\otimes u^{k}v^{l}\tilde{h}^{kl}_1 \in Z^1(A, A\#H)_1
	\end{equation}
	for any cohomological class in $\upH^1(A, A\#H)_1$. Applying $\partial^1$ to \eqref{eq:1-cocycle-11}, one obtains
	\[
	\sum \bigl((-1)^{j}quv^{j}+v^{j+1}\bigr)h^{0j}_1+\sum \bigl(u^{k+1}v^l+q^{-1}u^kv^{l+1}\bigr)\tilde{h}^{kl}_1=0.
	\]
	Immediately, we have $\tilde{h}^{kl}_1=0$ for all $k\geq 1$, and then
	\[
	\sum \bigl((-1)^{j}quv^{j}+v^{j+1}\bigr)h^{0j}_1+\sum \bigl(uv^l+q^{-1}v^{l+1}\bigr)\tilde{h}^{0l}_1=0.
	\]
	So for all odd $j$,  $h^{0j}_1=\tilde{h}^{0j}_1=0$, and for all even $j$, $qh^{0j}_1+\tilde{h}^{0j}_1=0$. Hence, the cocycle \eqref{eq:1-cocycle-11}  becomes
	\[
	\sum u^*\otimes v^{2j}h^{0,2j}_1+\sum v^*\otimes v^{2j}(-qh^{0,2j}_1)=\sum(u^*\otimes v^{2j}-v^*\otimes qv^{2j})h^{0,2j}_1.
	\]
	Therefore, $\upH^1(A, A\#H)_1=\bigoplus_{j} (u^*\otimes v^{2j}-v^*\otimes qv^{2j})H_1$.
\end{proof}

According to the proof, a byproduct is that
\begin{equation}\label{eq:byproduct-1}
(u^* \otimes u^{2i}v^{2j}-v^*\otimes qu^{2i}v^{2j})h_1=(u^* \otimes q^{-2i}v^{2i+2j}-v^*\otimes q^{-2i+1}v^{2i+2j})h_1
\end{equation}
holds true in $\upH^1(A, A\#H)_1$. The equation will be useful later on.

\begin{lem}\label{lem:h2-a}
	$\upH^2(A, A\#H)_0=(u^*v^*\otimes 1)H_0\oplus\bigoplus_{i,j} (u^*v^*\otimes u^{2i+1}v^{2j+1})H_0$, and $\upH^2(A, A\#H)_1=(u^*v^*\otimes 1)H_1\oplus \bigoplus_{j} (u^*v^*\otimes v^{2j+1})H_1$.
\end{lem}

\begin{proof}
	As in the proof of the previous lemma, we firstly compute $\upH^2(A, A\#H)_0$. Since
	\begin{gather*}
	\partial^1(u^*\otimes u^sv^th_0)=u^*v^*\otimes ((-1)^s+1)u^sv^{t+1}h_0, \\
	\partial^1(v^*\otimes u^sv^th_0)=u^*v^*\otimes (1+(-1)^t)u^{s+1}v^th_0,
	\end{gather*}
	any cohomological class in $\upH^2(A, A\#H)_0$ has a unique representative of the form
	\[
	u^*v^*\otimes h^{00}_0+\sum u^*v^*\otimes u^{2i+1}v^{2j+1}h^{2i+1,2j+1}_0
	\]
	which does not belong to $\Ima\partial^1$. Thus $\upH^2(A, A\#H)_0=(u^*v^*\otimes 1)H_0\oplus\bigoplus_{i,j} (u^*v^*\otimes u^{2i+1}v^{2j+1})H_0$.
	
	Next, let us determine $\upH^2(A, A\#H)_1$ using an analogous manner. Since
	\[
	\partial^1(v^*\otimes u^sv^th_1)=u^*v^*\otimes u^{s+1}v^th_1+u^*v^*\otimes u^sv^{t+1}q^{-1}h_1,
	\]
	we have
	\begin{align*}
	u^*v^*\otimes u^{i}v^jh_1&\equiv u^*v^*\otimes u^{i-1}v^{j+1}(-q^{-1})h_1 \\
	&\equiv u^*v^*\otimes u^{i-2}v^{j+2}q^{-2}h_1 \\
	&\;\;\vdots \\
	&\equiv u^*v^*\otimes v^{i+j}(-1)^iq^{-i}h_1\pmod{\Ima\partial^1}.
	\end{align*}
	On the other hand, it follows from
	\[
	\partial^1(u^*\otimes u^sv^th_1)=u^*v^*\otimes u^sv^{t+1}(-1)^sh_1+u^*v^*\otimes u^{s+1}v^{t}(-1)^tqh_1,
	\]
	that
	\begin{align*}
	u^*v^*\otimes u^{i}v^jh_1&\equiv u^*v^*\otimes u^{i-1}v^{j+1}(-1)^{i+j}q^{-1}h_1 \\
	&\equiv u^*v^*\otimes u^{i-2}v^{j+2}q^{-2}h_1 \\
	&\;\;\vdots \\
	&\equiv u^*v^*\otimes v^{i+j}(-1)^{i(i+j)}q^{-i}h_1 \\
	&\equiv u^*v^*\otimes v^{i+j}(-1)^{i+ij}q^{-i}h_1 \pmod{\Ima\partial^1}.
	\end{align*}
	So as cohomological classes, $u^*v^*\otimes u^{i}v^jh_1=u^*v^*\otimes v^{i+j}(-1)^iq^{-i}h_1$; they represent the trivial class when $ij$ is odd. Equivalently, $u^*v^*\otimes v^jh_1$ is nontrivial if $j$ is not the sum of two odd natural numbers, namely, $j=0$ or $j$ is odd. Therefore, 
	$\upH^2(A, A\#H)_1=(u^*v^*\otimes 1)H_1\oplus \bigoplus_{j} (u^*v^*\otimes v^{2j+1})H_1$.
\end{proof}

We emphasize the equation
\begin{equation}\label{eq:byproduct-2}
u^*v^*\otimes u^{i}v^jh_1=u^*v^*\otimes (-1)^iq^{-i}v^{i+j}h_1\in \upH^2(A, A\#H)_1,
\end{equation}
which is obtained during the proof of Lemma \ref{lem:h2-a}, as a second byproduct in this subsection.

\subsection{Computation of $\upHH^\bullet(A\# H)$}

In this subsection, let us determine the right $H$-action on $\upH^\bullet(A, A\# H)$, and then compute $\upHH^\bullet(A\# H)$ by seeking the $H$-invariants. It is well-known that the $H$-invariants are exactly the images of the endormorphism on $\upH^\bullet(A, A\# H)$ mapping $m$ to $m\blacktriangleleft \int$, where $\int$ is any nonzero integral of $H$ (necessarily, $\eps(\int)\neq 0$).

We choose $\int$ as the average of all elements of $\mathcal{B}$, which equals
\[
\frac{1}{8}(1+x)(1+y)(1+z).
\]
Note that $z^4=(z^2)^2=\frac{1}{4}(1+x+y-xy)^2=1$, so $z$ is invertible with $z^{-1}=z^3$. Thus by the relations of $H$, $z$ is a normal regular element so that there is an automorphism $\tau\colon H\to H$, $a\mapsto \tau(a):=\widehat{a}$ satisfying $za=\widehat{a}z$. Obviously, $\widehat{x}=y$, $\widehat{y}=x$, $\widehat{z}=z$.

By induction on $s$ and $t$,  it is easy to show that $z\triangleright u^{s}=q^{-s}v^{s}$ and $z\triangleright v^{t}=q^{t}u^{t}$. Hence
\begin{align*}
z\triangleright u^{s}v^t&=\frac{1}{2}\bigl((z\triangleright u^{s})(z\triangleright v^{t})+(z\triangleright u^{s})(xz\triangleright v^{t})+(yz\triangleright u^{s})(z\triangleright v^{t}) \\
&\mathrel{\phantom{=}} {}-(yz\triangleright u^{s})(xz\triangleright v^{t})\bigr) \\
&=(z\triangleright u^{s})(z\triangleright v^{t})=q^{-s}v^{s}q^{t}u^{t}=(-1)^{st}q^{t-s}u^tv^s.
\end{align*}
Besides, $x\triangleright u^{s}v^t=y\triangleright u^{s}v^t=u^{s}v^t$. Therefore, in $A\#H$, we have $xu^{s}v^t=u^{s}v^tx$, $yu^{s}v^t=u^{s}v^ty$, and
\begin{align*}
zu^{s}v^t&=\frac{1}{2}\bigl((z\triangleright u^{s}v^t)z+(z\triangleright u^{s}v^t)xz+(yz\triangleright u^{s}v^t)z-(yz\triangleright u^{s}v^t)xz\bigr) \\
&=(z\triangleright u^{s}v^t)z=(-1)^{st}q^{t-s}u^tv^sz.
\end{align*}

The $H$-action on $A^!$ is determined by
\begin{alignat*}{2}
u^*\triangleleft x&=u^*, &\quad v^*\triangleleft x&=v^*, \\
u^*\triangleleft y&=u^*, &\quad v^*\triangleleft y&=v^*, \\
u^*\triangleleft z&=qv^*, &\quad v^*\triangleleft z&=q^{-1}u^*,
\end{alignat*}
thus for any $\xi\in A^!$,
\begin{align*}
(\xi\otimes u^sv^th_0)\blacktriangleleft x&=\xi\triangleleft x\otimes xu^sv^th_0x=\xi\otimes u^sv^txh_0x=\xi\otimes u^sv^th_0, \\
(\xi\otimes u^sv^th_1)\blacktriangleleft x&=\xi\triangleleft x\otimes xu^sv^th_1x=\xi\otimes u^sv^txh_1x=\xi\otimes u^sv^th_1xy, \\
(\xi\otimes u^sv^th_0)\blacktriangleleft y&=\xi\triangleleft y\otimes yu^sv^th_0y=\xi\otimes u^sv^tyh_0y=\xi\otimes u^sv^th_0, \\
(\xi\otimes u^sv^th_1)\blacktriangleleft y&=\xi\triangleleft y\otimes yu^sv^th_1y=\xi\otimes u^sv^tyh_1y=\xi\otimes u^sv^th_1xy,
\end{align*}
namely, $x$ and $y$ act identically on $A^!\otimes (A\#H)_0$, and act on $A^!\otimes (A\#H)_1$ via multiplying by $xy$ from the right side. Furthermore,
\begin{align*}
z_{(1)}\otimes z_{(2)}&\otimes z_{(3)}=\frac{1}{4}(z\otimes z\otimes z+z\otimes xz\otimes z+yz\otimes z\otimes z-yz\otimes xz\otimes z \\
&\mathrel{\phantom{=}} {}+z\otimes z\otimes xz+z\otimes xz\otimes xz+yz\otimes z\otimes xz-yz\otimes xz\otimes xz \\
&\mathrel{\phantom{=}} {}+yz\otimes yz\otimes z+yz\otimes xyz\otimes z+z\otimes yz\otimes z-z\otimes xyz\otimes z \\
&\mathrel{\phantom{=}} {}-yz\otimes yz\otimes xz-yz\otimes xyz\otimes xz-z\otimes yz\otimes xz+z\otimes xyz\otimes xz).
\end{align*}
Since each $z_{(2)}$ tensor summand satisfies $\xi\triangleleft z_{(2)}=\xi\triangleleft z$, one has
\begin{align*}
(\xi\otimes u^sv^th_0)\blacktriangleleft z&=\xi\triangleleft z\otimes \frac{1}{2}(zu^sv^th_0z+zu^sv^th_0xz+zyu^sv^th_0z-zyu^sv^th_0xz) \\
&=\xi\triangleleft z\otimes \frac{1}{2}(zu^sv^th_0z+zu^sv^th_0xz+zu^sv^tyh_0z-zu^sv^tyh_0xz) \\
&=\xi\triangleleft z\otimes \frac{1}{2}zu^sv^th_0(1+x+y-xy)z \\
&=\xi\triangleleft z\otimes(-1)^{st}q^{t-s}u^tv^szh_0z^{-1} \\
&=\xi\triangleleft z\otimes(-1)^{st}q^{t-s}u^tv^s\widehat{h_0}, \\
(\xi\otimes u^sv^th_1)\blacktriangleleft z&=\xi\triangleleft z\otimes \frac{1}{2}(zu^sv^th_1z+zu^sv^th_1xz+zyu^sv^th_1z-zyu^sv^th_1xz) \\
&=\xi\triangleleft z\otimes \frac{1}{2}(zu^sv^th_1z+zu^sv^th_1xz+zu^sv^tyh_1z-zu^sv^tyh_1xz) \\
&=\xi\triangleleft z\otimes \frac{1}{2}zu^sv^t(h_1+h_1x+h_1x-h_1)z \\
&=\xi\triangleleft z\otimes (-1)^{st}q^{t-s}u^tv^szh_1xz \\
&=\xi\triangleleft z\otimes (-1)^{st}q^{t-s}u^tv^s\widehat{h_1}yz^2.
\end{align*}

By using the foregoing formulas, let us begin to compute $\upH^\bullet(A, A\#H)\blacktriangleleft\int$.

\begin{lem}\label{lem:invariant}
	For any $\xi\in A^!$, $h_0\in H_0$ and $h_1\in H_1$, one has
	\begin{align*}
	(\xi\otimes u^sv^th_0)\blacktriangleleft \tint&=\bigl(\xi\otimes u^sv^th_0+\xi\triangleleft z\otimes (-1)^{st}q^{t-s}u^tv^s\widehat{h_0}\bigr)\frac{1}{2}, \\
	(\xi\otimes u^sv^th_1)\blacktriangleleft \tint&=\bigl(\xi\otimes u^sv^th_1+\xi\triangleleft z\otimes (-1)^{st}q^{t-s}u^tv^s\widehat{h_1}\bigr)\frac{1+xy}{4}.
	\end{align*}
\end{lem}

\begin{proof}
	Since $x$, $y$ act identically on $A^!\otimes (A\#H)_0$, it is routine to show
    \[
    (\xi\otimes u^sv^th_0)\blacktriangleleft \int=(\xi\otimes u^sv^th_0)\blacktriangleleft \frac{1+z}{2}.
    \]
    So the first assertion is true.
	
	For the second, we have
	\[
	\biggl((\xi\otimes u^sv^th_1)\blacktriangleleft\frac{1+x}{2}\biggr)\blacktriangleleft\frac{1+y}{2} =\xi\otimes u^sv^th_1 \biggl(\frac{1+xy}{2}\biggr)^2=\xi\otimes u^sv^th_1 \frac{1+xy}{2},
	\]
	so that
	\begin{align*}
	(\xi\otimes u^sv^th_1)\blacktriangleleft \tint&=\biggl(\xi\otimes u^sv^th_1 \frac{1+xy}{2}\biggr)\blacktriangleleft\frac{1+z}{2} \\
	&=\xi\otimes u^sv^th_1 \frac{1+xy}{4}+\xi\triangleleft z\otimes (-1)^{st}q^{t-s}u^tv^s\widehat{h_1}\frac{1+xy}{4}yz^2.
	\end{align*}
	The second assertion follows from
	\[
	\frac{1+xy}{4}yz^2=\frac{x+y}{4}z^2=\frac{(x+y)(1+x+y-xy)}{8}=\frac{1+xy}{4}. \qedhere
	\]
\end{proof}

\begin{thm}\label{thm:hoch-cohomology-smash}
	The Hochschild cohomological groups of $A\#H$, as vector spaces, admit bases as follows:
 \begin{alignat*}{2}
	\upHH^0(A\#H): \;
       &(1\otimes q^{2i}u^{2i}v^{2j}+1\otimes q^{2j}u^{2j}v^{2i})\frac{1+xy}{2}, &  &i\leq j, \\
       & (1\otimes q^{2i}u^{2i}v^{2j}+1\otimes q^{2j}u^{2j}v^{2i})\frac{1-xy}{2}, &  &i\leq j, \\
       &(1\otimes q^{2i}u^{2i}v^{2j}+1\otimes q^{2j}u^{2j}v^{2i})\frac{x+y}{2}, &  &i\leq j, \\
       & (1\otimes q^{2i}u^{2i}v^{2j}-1\otimes q^{2j}u^{2j}v^{2i})\frac{x-y}{2}, &  &i<j; \\
	\upHH^1(A\#H): \;
     &(u^*\otimes q^{2i}u^{2i+1}v^{2j}+v^*\otimes q^{2j}u^{2j}v^{2i+1})\frac{1+xy}{2}, &  &i, j\geq 0, \\
	&(u^*\otimes q^{2i}u^{2i+1}v^{2j}+v^*\otimes q^{2j}u^{2j}v^{2i+1})\frac{1-xy}{2}, &  &i, j\geq 0, \\
	&(u^*\otimes q^{2i}u^{2i+1}v^{2j}+v^*\otimes q^{2j}u^{2j}v^{2i+1})\frac{x+y}{2}, &  &i, j\geq 0, \\
	 &(u^*\otimes q^{2i}u^{2i+1}v^{2j}-v^*\otimes q^{2j}u^{2j}v^{2i+1})\frac{x-y}{2}, &  &i, j\geq 0; \\
	\upHH^2(A\#H): \;
    &(u^*v^*\otimes q^{2i}u^{2i+1}v^{2j+1}-u^*v^*\otimes q^{2j}u^{2j+1}v^{2i+1})\frac{1+xy}{2}, &  &i< j, \\
	&(u^*v^*\otimes q^{2i}u^{2i+1}v^{2j+1}-u^*v^*\otimes q^{2j}u^{2j+1}v^{2i+1})\frac{1-xy}{2}, &  &i< j, \\
	 &(u^*v^*\otimes q^{2i}u^{2i+1}v^{2j+1}-u^*v^*\otimes q^{2j}u^{2j+1}v^{2i+1})\frac{x+y}{2}, &  &i< j, \\
	&(u^*v^*\otimes q^{2i}u^{2i+1}v^{2j+1}+u^*v^*\otimes q^{2j}u^{2j+1}v^{2i+1})\frac{x-y}{2}, &  &i\leq j, \\
	& u^*v^*\otimes k, \;\; k\in\biggl\{\frac{1+xy}{2}, \frac{1-xy}{2}, \frac{x+y}{2}, \frac{z+xyz}{2}, \frac{xz+yz}{2}\biggr\}.
	\end{alignat*}
\end{thm}

\begin{proof}
	Recall that the automorphism $\tau\colon H\to H$, $a\mapsto \widehat{a}$ has two eigenvalues, $1$ and $-1$, and the corresponding eigenspaces are spanned by bases $\{1\pm xy, x+y, z\pm xyz, xz+yz\}$ and $\{x-y,xz-yz\}$ respectively. In order to prove the theorem, it is sufficient to apply Lemmas \ref{lem:h0-a}, \ref{lem:h1-a}, \ref{lem:h2-a} and \ref{lem:invariant}.
	
	Since the proof is completely computational, we only verify the basis of $\upHH^2(A\#H)$ here. The situations for $\upHH^0(A\#H)$ and $\upHH^1(A\#H)$ are left to the reader.
	
	A direct computation yields
	\begin{align*}
	u^*v^*\triangleleft z&=\frac{1}{2}\bigl((u^*\triangleleft z)(v^*\triangleleft z)+(u^*\triangleleft z)(v^*\triangleleft xz)+(u^*\triangleleft yz)(v^*\triangleleft z) \\
	&\mathrel{\phantom{=}} {}-(u^*\triangleleft yz)(v^*\triangleleft xz)\bigr) \\
	&=(u^*\triangleleft z)(v^*\triangleleft z)=qv^*q^{-1}u^*=u^*v^*,
	\end{align*}
	so by Lemma \ref{lem:invariant}, for any $h_0\in H_0$,
	\[
	(u^*v^*\otimes u^{2i+1}v^{2j+1}h_0)\blacktriangleleft \tint=u^*v^*\otimes (u^{2i+1}v^{2j+1}h_0- q^{2j-2i}u^{2j+1}v^{2i+1}\widehat{h_0})\dfrac{1}{2}.
	\]
	Now let $h_0$ be a basis element of $H_0$. When $h_0$ is one of $1+xy$, $1-xy$, $x+y$, we have $\widehat{h_0}=h_0$; when $h_0$ is $x-y$,  $\widehat{h_0}=-h_0$. Hence, up to scalars, we obtain a part of the basis of $\upHH^2(A\#H)$ as desired.
	
	Also, by \eqref{eq:byproduct-2}, we have
	\begin{align*}
	(u^*v^*\otimes v^{2j+1}h_1)\blacktriangleleft \tint 	&=(u^*v^*\otimes v^{2j+1}h_1+u^*v^*\otimes q^{2j+1}u^{2j+1} \widehat{h_1})\frac{1+xy}{4} \\
	&=(u^*v^*\otimes v^{2j+1}h_1-u^*v^*\otimes v^{2j+1} \widehat{h_1})\frac{1+xy}{4},
	\end{align*}
	which is identically zero for all $h_1$ corresponding to the eigenvalue $1$. On the other hand, when  $h_1=xz-yz$,   since
	\[
	(xz-yz)\frac{1+xy}{4}=\frac{(x-y)(1+xy)}{4}z=0,
	\]
	we conclude that no new base element is produced from $(u^*v^*\otimes v^{2j+1})H_1$.
	
	Finally, we have to consider $(u^*v^*\otimes k)\blacktriangleleft\int$ with $k\in H$. This is very easy, and we find the five exceptional basis elements as listed in the theorem.
\end{proof}

\subsection{Multiplication table of $\upHH^\bullet(A\#H)$}

In order to give the cup product on $\upHH^\bullet(A\#H)$, we denote the bases of $\upHH^\bullet(A\#H)$ listed in Theorem \ref{thm:hoch-cohomology-smash} by
$\varepsilon_{1}^{i,j}$, $\varepsilon_{2}^{i,j}$, $\varepsilon_{3}^{i,j}$, $\varepsilon_{4}^{i,j}$, $\eta_{1}^{i,j}$, $\eta_{2}^{i,j}$, $\eta_{3}^{i,j}$, $\eta_{4}^{i,j}$, $\omega_{1}^{i,j}$, $\omega_{2}^{i,j}$, $\omega_{3}^{i,j}$, $\omega_{4}^{i,j}$, $\omega'_{1}$, $\omega'_{2}$, $\omega'_{3}$, $\omega''_{1}$, $\omega''_{3}$ successively. For convenience, we extend the notations as follows: for $i<j$,
$\varepsilon_{1}^{j,i}\triangleq\varepsilon_{1}^{i,j}$, $\varepsilon_{2}^{j,i}\triangleq\varepsilon_{2}^{i,j}$, $\varepsilon_{3}^{j,i}\triangleq\varepsilon_{3}^{i,j}$, $\varepsilon_{4}^{j,i}\triangleq-\varepsilon_{4}^{i,j}$,
$\omega_{1}^{j,i}\triangleq-\omega_{1}^{i,j}$, $\omega_{2}^{j,i}\triangleq-\omega_{2}^{i,j}$, $\omega_{3}^{j,i}\triangleq-\omega_{3}^{i,j}$, $\omega_{4}^{j,i}\triangleq\omega_{4}^{i,j}$, that is, all the new notations are symmetric with respect to $i$ and $j$, except that $\varepsilon_{4}^{i,j}$, $\omega_{1}^{i,j}$,
$\omega_{2}^{i,j}$
and $\omega_{3}^{i,j}$ are anti-symmetric.

Notice that 
the cup product $\smallsmile$ on $\upHH^\bullet(A\#H)$ is induced by the multiplication on $A^!\otimes (A\#H)$. We list the operation in Tables~\ref{cup-product1}-\ref{cup-product4} as follows. We remind the reader that $(\upHH^\bullet(A\#H),\smallsmile)$ is a graded commutative algebra, so the four tables are enough.

\begin{table}[htbp]
\centering\caption{\label{cup-product1}Cup product between $\upHH^0(A\#H)$ and $\upHH^0(A\#H)$}
\scalebox{0.9}{
\begin{tabular}{|p{0.8cm}<{\centering}|p{2.8cm}<{\centering}|p{2.8cm}<{\centering}|p{2.8cm}<{\centering}|p{2.8cm}<{\centering}|}
\hline
 $\smallsmile$ & $\varepsilon_{1}^{s,t}$ & $\varepsilon_{2}^{s,t}$ & $\varepsilon_{3}^{s,t}$ & $\varepsilon_{4}^{s,t}$ \\
 \hline
 $\varepsilon_{1}^{i,j}$ & $\varepsilon_{1}^{i+s,j+t}+\varepsilon_{1}^{i+t,j+s}$ & $0$ & $\varepsilon_{3}^{i+s,j+t}+\varepsilon_{3}^{i+t,j+s}$ & $0$
\\
\hline
$\varepsilon_{2}^{i,j}$ & $0$ & $\varepsilon_{2}^{i+s,j+t}+\varepsilon_{2}^{i+t,j+s}$   & $0$  &  $\varepsilon_{4}^{i+s,j+t}+\varepsilon_{4}^{i+t,j+s}$
\\
\hline
 $\varepsilon_{3}^{i,j}$ & $\varepsilon_{3}^{i+s,j+t}+\varepsilon_{3}^{i+t,j+s}$ & $0$ & $\varepsilon_{1}^{i+s,j+t}+\varepsilon_{1}^{i+t,j+s}$ & $0$
\\
\hline
 $\varepsilon_{4}^{i,j}$ & $0$  &$\varepsilon_{4}^{i+s,j+t}+\varepsilon_{4}^{i+t,j+s}$ & $0$  &$\varepsilon_{2}^{i+s,j+t}+\varepsilon_{2}^{i+t,j+s}$
 \\
  \hline
\end{tabular}}
\end{table}
\vskip 2pt

\begin{table}[htbp]
\centering\caption{\label{cup-product2}Cup product between $\upHH^1(A\#H)$ and $\upHH^0(A\#H)$}
\scalebox{0.9}{
\begin{tabular}{|p{0.8cm}<{\centering}|p{2.8cm}<{\centering}|p{2.8cm}<{\centering}|p{2.8cm}<{\centering}|p{2.8cm}<{\centering}|}
\hline
 $\smallsmile$ & $\varepsilon_{1}^{s,t}$ & $\varepsilon_{2}^{s,t}$ & $\varepsilon_{3}^{s,t}$ & $\varepsilon_{4}^{s,t}$\\
 \hline
 $\eta_{1}^{i,j}$ & $\eta_{1}^{i+s,j+t}+\eta_{1}^{i+t,j+s}$ & $0$ & $\eta_{3}^{i+s,j+t}+\eta_{3}^{i+t,j+s}$ & $0$
\\
\hline
$\eta_{2}^{i,j}$ & $0$ & $\eta_{2}^{i+s,j+t}+\eta_{2}^{i+t,j+s}$   & $0$  &  $\eta_{4}^{i+s,j+t}+\eta_{4}^{i+t,j+s}$
\\
\hline
 $\eta_{3}^{i,j}$ & $\eta_{3}^{i+s,j+t}+\eta_{3}^{i+t,j+s}$ & $0$ & $\eta_{1}^{i+s,j+t}+\eta_{1}^{i+t,j+s}$ & $0$
\\
\hline
 $\eta_{4}^{i,j}$ & $0$  &$\eta_{4}^{i+s,j+t}+\eta_{4}^{i+t,j+s}$ & $0$  &$\eta_{2}^{i+s,j+t}+\eta_{2}^{i+t,j+s}$
\\
  \hline
\end{tabular}}
\end{table}
\vskip 2pt

\begin{table}[htbp]
\centering\caption{\label{cup-product3}Cup product between $\upHH^2(A\#H)$ and $\upHH^0(A\#H)$}
\scalebox{0.9}{
\begin{tabular}{|p{0.8cm}<{\centering}|p{2.8cm}<{\centering}|p{2.8cm}<{\centering}|p{2.8cm}<{\centering}|p{2.8cm}<{\centering}|}
\hline
$\smallsmile$ & $\varepsilon_{1}^{s,t}$ & $\varepsilon_{2}^{s,t}$ & $\varepsilon_{3}^{s,t}$ & $\varepsilon_{4}^{s,t}$ \\
 \hline
 $\omega_{1}^{i,j}$ & $\omega_{1}^{i+s,j+t}+\omega_{1}^{i+t,j+s}$ & $0$ & $\omega_{3}^{i+s,j+t}+\omega_{3}^{i+t,j+s}$ & $0$
\\
\hline
$\omega_{2}^{i,j}$ & $0$ & $\omega_{2}^{i+s,j+t}+\omega_{2}^{i+t,j+s}$   & $0$  &  $\omega_{4}^{i+s,j+t}-\omega_{4}^{i+t,j+s}$
\\
\hline
 $\omega_{3}^{i,j}$ & $\omega_{3}^{i+s,j+t}+\omega_{3}^{i+t,j+s}$ & $0$ & $\omega_{1}^{i+s,j+t}+\omega_{1}^{i+t,j+s}$ & $0$
\\
\hline
 $\omega_{4}^{i,j}$ & $0$  &$\omega_{4}^{i+s,j+t}+\omega_{4}^{i+t,j+s}$ & $0$  &$\omega_{2}^{i+s,j+t}-\omega_{2}^{i+t,j+s}$
\\
  \hline
   $\omega'_{1}$ & $2\delta_{s+t}^0\omega'_{1}$  & $0$ & $2\delta_{s+t}^0\omega'_{3}$ & $0$
\\
  \hline
   $\omega'_{2}$ & $0$  &$2\delta_{s+t}^0\omega'_{2}$ & $0$  &$0$
\\
  \hline
   $\omega'_{3}$ & $2\delta_{s+t}^0\omega'_{3}$  & $0$   &$2\delta_{s+t}^0\omega'_{1}$ &$0$
\\
  \hline
   $\omega''_{1}$  &$2\delta_{s+t}^0\omega''_{1}$ & $0$  &$2\delta_{s+t}^0\omega''_{3}$ & $0$
\\
  \hline
   $\omega''_{3}$ & $2\delta_{s+t}^0\omega''_{3}$  &$0$ &$2\delta_{s+t}^0\omega''_{1}$ & $0$
\\
  \hline
\end{tabular}}
\end{table}

\begin{table}[htbp]
\centering\caption{\label{cup-product4}Cup product between $\upHH^1(A\#H)$ and $\upHH^1(A\#H)$}
\scalebox{0.9}{
\begin{tabular}{|p{0.8cm}<{\centering}|p{2.8cm}<{\centering}|p{2.8cm}<{\centering}|p{2.8cm}<{\centering}|p{2.8cm}<{\centering}|}
\hline
 $\smallsmile$ &$\eta_{1}^{s,t}$ & $\eta_{2}^{s,t}$ & $\eta_{3}^{s,t}$ & $\eta_{4}^{s,t}$\\
 \hline
 $\eta_{1}^{i,j}$ & $\omega_{1}^{i+t,j+s}$ & $0$ & $\omega_{3}^{i+t,j+s}$ & $0$
\\
\hline
$\eta_{2}^{i,j}$ & $0$ & $\omega_{2}^{i+t,j+s}$   & $0$  &  $-\omega_{4}^{i+t,j+s}$
\\
\hline
 $\eta_{3}^{i,j}$ & $\omega_{3}^{i+t,j+s}$ & $0$ & $\omega_{1}^{i+t,j+s}$ & $0$
\\
\hline
 $\eta_{4}^{i,j}$ & $0$  &$\omega_{4}^{i+t,j+s}$ & $0$  &$-\omega_{2}^{i+t,j+s}$
\\
  \hline
\end{tabular}}
\end{table}

Before ending the section, we mention that the four tables are obtained by direct computation. Observe that the basis elements we chosen are of the form $\sum_i \xi_i\otimes a_ih_i$ with $h_i\in H_0$ except $\omega''_1$, $\omega''_3$, and that $H_0$ is spanned by $1$, $x$, $y$, $xy$, which are all group-like and act on $A$ trivially. So the cup product is not hard to compute, and hence we omit all the details.

\vskip 10pt

\noindent {\bf Acknowledgments.}\quad
This work is supported by the National Natural Science Foundation of China (Nos. 11971418, 11871125 and 11771085). The authors are grateful to Qi Lou, Shuanhong Wang, Guodong Zhou, and Ruipeng Zhu for helpful conversations.


\begin{thebibliography}{99}

\bibitem{BW22} {\sc B. Briggs and S. Witherspoon}, {\em Hochschild cohomology of twisted tensor products}, Math. Z. {\bf 301} (2022), 1237--1257.

\bibitem{BW07} {\sc S.M. Burciu and S. Witherspoon}, {\em Hochschild cohomology of smash products and rank one Hopf algebras}, Bibioteca de la Revista Matematica Iberoamericana Actas del, (2007), 153--170.

\bibitem{CCMT07}{\sc S. Caenepeel, S, Crivei, A. Marcus, and M. Takeuchi}, {\em Morita equivalences induced by bimodules over Hopf--Galois extensions}, J. Algebra {\bf 314} (2007), 267--302.


\bibitem{Far05} {\sc M. Farinati}, {\em Hochschild duality, localization, and smash products}, J. Algebra {\bf 284}(1) (2005), 415--434.

\bibitem{Ger63} {\sc M. Gerstenhaber}, {\em The cohomology structure of an associative ring}, Ann. Math. {\bf 78} (1963), 267--288.

\bibitem{GK04} {\sc V. Ginzburg and D. Kaledin}, {\em Poisson deformation of symplectic quotient singularities}, Adv. Math. {\bf 186}(1) (2004), 1--57.

\bibitem{Hoc45} {\sc G. Hochschild}, {\em On the cohomology groups of an associative algebra}, Ann. Math. {\bf 46} (1945), 58--67.

\bibitem{KP66} {\sc G. Kac and V. Paljutkin}, {\em Finite ring groups}, Trudy Moskov. Mat. Obšč. {\bf 15} (1966), 224--261.


\bibitem{Kap75} {\sc I. Kaplansky}, Bialgebras, University of Chicago Lecture Notes, 1975.

\bibitem{KMO21} {\sc T. Karada\v{g}, D. McPhate, P. S. Ocal, T. Oke, and S. Witherspoon}, {\em Gerstenhaber brackets on Hochschild cohomology of general twisted tensor products}, J. Pure Appl. Algebra {\bf 225}(6) (2021), 106597.

\bibitem{Kay07} {\sc A. Kaygun}, {\em Hopf--Hochschild (co)homology of module algebras}, Homology, Homotopy Appl. {\bf 9} (2007), 451--472.

\bibitem{KT81} {\sc H.F. Kreimer and M. Takeuchi}, {\em Hopf algebras and Galois extensions of an algebra}, Indiana Univ. Math. J. {\bf 30} (1981), 675--692.

\bibitem{LWZ12} {\sc L.-Y. Liu, Q.-S. Wu, and C. Zhu}, {\em Hopf action on Calabi-Yau algebras}, New Trends in Noncommutative Algebra, 189--209, Contemp. Math., vol. 562, Amer. Math. Soc., Providence, RI, 2012.

\bibitem{Mas95} {\sc A. Masuoka}, {\em Semisimple Hopf algebras of dimension 6, 8}, Israel J. Math. {\bf 92}(1) (1995), 361--373.

\bibitem{Mas08} {\sc A. Masuoka}, {\em Classification of semisimple Hopf algebras}, Handbook of Algebra, {\bf 5} (2008),  429--455.

\bibitem{Mon93} {\sc S. Montgomery}, Hopf Algebras and Their Actions on Rings, CBMS Regional Conference Series in Mathematics, vol. 82, American Mathematical Society, 1993.

\bibitem{NSW11} {\sc D. Naidu, P. Shroff, and S. Witherspoon}, {\em Hochschild cohomology of group extensions of quantum symmetric algebras}, Proc. Amer. Math. Soc. {\bf 139}(5) (2011), 1553--1567.

\bibitem{Neg15} {\sc C. Negron}, {\em Spectral sequences for the cohomology rings of a smash product}, J. Algebra {\bf 433} (2015), 73--106.

\bibitem{Neg17} {\sc C. Negron}, {\em The cup product on Hochschild cohomology via twisting cochains and applications to Koszul rings}, J. Pure Appl. Algebra {\bf 221} (2017), 1112--1133.

\bibitem{Neg19} {\sc C. Negron}, {\em Braided Hochschild cohomology and Hopf actions}, J. Noncommut. Geom. {\bf 13} (2019), 1--33.

\bibitem{NW17} {\sc C. Negron and S. Witherspoon}, {\em The Gerstenhaber bracket as a Schouten bracket for polynomial rings extended by finite groups}, Proc. Lond. Math. Soc. {\bf 115} (2017), 1149--1169.

\bibitem{SW12} {\sc A.V. Shepler and S. Witherspoon}, {\em Group actions on algebras and the graded Lie structure of Hochschild cohomology}, J. Algebra {\bf 351} (2012), 350--381.

\bibitem{SW11} {\sc A.V. Shepler and S. Witherspoon}, {\em Quantum differentiation and chain maps of bimodule complexes}, Algebra Number Theory {\bf 5} (2011), 339--360.

\bibitem{SWi11} {\sc A.V. Shepler and S. Witherspoon}, {\em Finite groups acting linearly: Hochschild cohomology and the cup product}, Adv. Math. {\bf 226}(4) (2011), 2884--2910.

\bibitem{Ste95} {\sc D. Stefan}, {\em Hochschild cohomology of Hopf Galois extensions}, J. Pure Appl. Algebra {\bf 103} (1995), 221--233.

\bibitem{Ste99} {\sc D. Stefan}, {\em Hopf algebras of low dimension}, J. Algebra {\bf 211} (1999), 343--366.

\bibitem{Vol} {\sc Y. Volkov}, {\em Gerstenhaber bracket on the Hochschild cohomology via an arbitrary resolution}, Proc. Edinburgh Math. Soc. {\bf 62}(3) (2019), 817--836.

\bibitem{Wam93} {\sc M. Wambst}, {\em Complexes de Koszul quantiques}, Ann. Fourier {\bf 43}(4) (1993), 1089--1156.

\bibitem{Wang18}{\sc Z. Wang}, {\em Note on the Hochschild cohomology of Hopf--Galois extension}, Acta Math. Hungar. {\bf 154}(1) (2018), 223--230.

\bibitem{WZ16} {\sc S. Witherspoon and G.D. Zhou}, {\em Gerstenhaber brackets on Hochschild cohomology of quantum symmetric algebras and their group extensions}, Pacific J. Math.  {\bf 283}(1) (2016), 223--255.

\end{thebibliography}
\end{document}